\documentclass[letterpaper,11pt]{amsart}
\usepackage{amsmath,amssymb,amsthm}
\usepackage{fullpage}
\usepackage[bookmarks=true,colorlinks,linkcolor=blue]{hyperref}
\usepackage{graphicx,subcaption,standalone}
\usepackage{mathrsfs}
\usepackage{tikz,tikz-cd}
\usetikzlibrary{positioning,knots,patterns,math,hobby}
\usetikzlibrary{decorations.markings,decorations.pathreplacing,decorations.pathmorphing}
\usepackage{ifthen,comment}
\allowdisplaybreaks

\tikzset{knot diagram/every knot diagram/.style={background color=gray!20,clip width=6,end tolerance=5pt,clip radius=0.2cm}}
\tikzset{edge/.style={line width=0.8}}
\tikzset{wall/.style={very thick}}

\newcommand{\picmargin}{\mathop{}\!}
\newcommand{\stsize}{\footnotesize}

\newcommand{\TanglePic}[5]{
\picmargin
\begin{tikzpicture}[baseline=(ref.base)]
\tikzmath{\xw=#1; \yh=#2; \yd=0; \yu=0;}
\ifthenelse{\equal{#3}{<-}\OR\equal{#4}{<-}}{\tikzmath{\yd=-0.1;}}{}
\ifthenelse{\equal{#3}{->}\OR\equal{#4}{->}}{\tikzmath{\yu=0.1;}}{}
\tikzmath{\yt=\yh+\yu;}
\fill[gray!20] (0,\yd)rectangle(\xw,\yt);
\node(ref) at ({\xw/2},{\yh/2}) {\phantom{$-$}};
\begin{scope}[wall]
\ifthenelse{\equal{#3}{w}}{
	\draw (0,\yd) --(0,\yt);
}{\ifthenelse{\equal{#3}{}}{}{
	\draw[#3] (0,\yd) --(0,\yt);
}}
\ifthenelse{\equal{#4}{w}}{
	\draw (\xw,\yd) --(\xw,\yt);
}{\ifthenelse{\equal{#4}{}}{}{
	\draw[#4] (\xw,\yd) --(\xw,\yt);
}}
\end{scope}
#5
\end{tikzpicture}
\picmargin
}

\newcommand{\HorizontalTangle}[5]{
\TanglePic{0.9}{0.9}{#1}{#2}{
\tikzmath{\ya=\yh/2+0.2; \yb=\yh/2-0.2;}
\path (0,\ya) coordinate (C) (0,\yb) coordinate (D);
\path (\xw,\ya) coordinate (A) (\xw,\yb) coordinate (B);
\ifthenelse{\equal{#3}{}\AND\equal{#4}{}}{}{
	\draw[right,inner sep=2pt] (A)node{\stsize #3} (B)node{\stsize #4};
}
#5
}}

\newcommand{\crosswall}[4]{
\HorizontalTangle{}{#1}{#3}{#4}{
\ifthenelse{\equal{#3}{a}}{
	\draw[edge] (C) ..controls ({\xw/2},\ya) and ({\xw/2},\yb).. (B);
	\draw[edge] (D) ..controls ({\xw/2},\yb) and ({\xw/2},\ya).. (A);
}{
	\begin{knot}
	\strand[edge] (C) ..controls ({\xw/2},\ya) and ({\xw/2},\yb).. (B);
	\strand[edge] (D) ..controls ({\xw/2},\yb) and ({\xw/2},\ya).. (A);
	\ifthenelse{\(\equal{#2}{n}\)}{\flipcrossings{1}}{}
	\end{knot}
}
\path[edge] (C) ..controls ({\xw/2},\ya) and ({\xw/2},\yb).. (B);
\path[edge] (D) ..controls ({\xw/2},\yb) and ({\xw/2},\ya).. (A);
}}

\newcommand{\twowall}[3]{
\HorizontalTangle{}{#1}{#2}{#3}{
\draw[edge] (C) -- (A);
\draw[edge] (D) -- (B);
}}

\newcommand{\circlediag}{\TanglePic{0.9}{0.9}{}{}{\draw[edge] (0.45,0.45) circle (0.25);}}

\newcommand{\capwall}[5][]{
\HorizontalTangle{}{#2}{#4}{#5}{
\draw[edge] (\xw,\yb) ..controls (0.1,\yb) and (0.1,\ya).. (\xw,\ya);
}}

\newcommand{\capnearwall}{
\HorizontalTangle{}{w}{}{}{
\draw[edge] (0,\yb) ..controls (0.7,\yb) and (0.7,\ya).. (0,\ya);
}}

\providecommand{\abs}[1]{{|{#1}|}}

\DeclareMathOperator{\lead}{lt}
\DeclareMathOperator{\order}{ord}

\DeclareMathOperator{\coker}{coker}
\DeclareMathOperator{\image}{im}
\DeclareMathOperator{\vol}{vol}
\DeclareMathOperator{\dego}{\deg^\circ}
\DeclareMathOperator{\maxspec}{MaxSpec}

\newcommand{\ints}{\mathbb{Z}}
\newcommand{\nats}{\mathbb{N}}
\newcommand{\reals}{\mathbb{R}}
\newcommand{\cx}{\mathbb{C}}

\newcommand{\surface}{\mathfrak{S}}
\newcommand{\surclose}{\overline{\surface}}
\newcommand{\marked}{\mathcal{P}}
\newcommand{\peripheral}{\mathring{\marked}}
\newcommand{\perialg}{R[\peripheral]}
\newcommand{\periext}{\perialg^\diamond}

\newcommand{\quasi}{\mathcal{E}}
\newcommand{\face}{\mathcal{F}}

\newcommand{\extY}{\bar{\mathcal{Y}}}

\newcommand{\rmmatch}{\mathrm{ma}}
\newcommand{\skein}{\mathscr{S}}
\newcommand{\matchedS}{\skein^\rmmatch}
\newcommand{\evenS}{\skein^\mathrm{ev}}
\newcommand{\balLambda}{\langle\Lambda\rangle}

\newcommand{\centV}{\mathcal{Z}}

\newcommand{\term}[1]{\textbf{#1}}

\newtheorem{theorem}{Theorem}[section]
\newtheorem*{thm}{Theorem}
\newtheorem{lemma}[theorem]{Lemma}
\newtheorem{proposition}[theorem]{Proposition}
\newtheorem{corollary}[theorem]{Corollary}
\theoremstyle{remark}
\newtheorem*{remark}{Remark}

\makeatletter
\renewcommand*{\fps@figure}{htb}
\makeatother

\usepackage{biblatex}
\begin{document}

\title{Center of the Stated Skein Algebra}
\author[Tao Yu]{Tao Yu}
\address{Shenzhen International Center for Mathematics, Southern University of Science and Technology, 1088 Xueyuan Avenue, Shenzhen, Guangdong, China}
\email{yut6@sustech.edu.cn}


\begin{abstract}
The stated skein algebra is a generalization of the Kauffman bracket skein algebra introduced in the study of quantum trace maps. When the quantum parameter is a root of unity, the stated skein algebra has a big center and is finitely generated as a module over the center. We give the center a simple description and calculate the dimension over center of the stated skein algebra.
\end{abstract}

\maketitle

\section{Introduction}

Let $R$ be a commutative domain with an invertible element $q^{1/2}$. The Kauffman bracket skein algebra $\mathring\skein_q(\surface)$ of a surface $\surface$, introduced by Przytycki \cite{Pr} and Turaev \cite{Tu}, is an $R$-algebra spanned by framed unoriented links in the thickened surface $\surface\times(-1,1)$ modulo the Kauffman bracket relations
\begin{equation*}
\TanglePic{1.1}{0.9}{}{}{\tikzmath{\xl=0.1;\xr=\xw-\xl;}
\begin{knot}
\strand[edge] (\xr,\yt)--(\xl,\yd); \strand[edge] (\xl,\yt)--(\xr,\yd);
\end{knot}}
=q\TanglePic{1.1}{0.9}{}{}{
\foreach \x in {0.1,\xw-0.1} \draw[edge] (\x,\yt)..controls (ref)..(\x,\yd);}
+q^{-1}\TanglePic{1.1}{0.9}{}{}{ \tikzmath{\xl=0.1;\xr=\xw-\xl;}
\foreach \y in {\yt,\yd} \draw[edge] (\xl,\y)..controls (ref)..(\xr,\y);},
\qquad
\circlediag=(-q^2-q^{-2})\TanglePic{0.9}{0.9}{}{}{}.
\end{equation*}
The Kauffman bracket skein algebra is connected to many areas in low dimensional topology. Quantizations of character varieties \cite{Bul,BFK,PS} and Teichm\"uller spaces \cite{BWqtr,Mu} are a few examples.

In \cite{LeTriang}, a generalization called the stated skein algebra $\skein_q(\surface)$. The surface here is of the form $\surface=\surclose\setminus\marked$, where $\surclose$ is an oriented compact surface, and $\marked$ is a finite set with at least one point in each boundary component of $\surclose$. The stated skein algebra is spanned by framed tangles with endpoint on the boundary of $\surface$, satisfying the addition relations given in Section~\ref{sec-stateS}. The original application was to give a simple construction of the quantum trace map, first defined by Bonahon and Wong \cite{BWqtr}. More related results have been developed since. See e.g. \cite{CL,Fa,Kor}.

We are interested in the center and dimension over center of the skein algebras. They have important applications in the representation theory. Suppose $A$ is a finitely generated $\cx$-algebra and $Z$ is the center of $A$. Any irreducible representation of $A$ restricts to a representation of $Z$ and, by Schur's lemma, defines an algebra homomorphism $Z\to\cx$. This is a point in the max spectrum $\centV:=\maxspec Z$. Further assume that $A$ is a domain and finitely generated as a $Z$-module. 
Let $\tilde{Z}$ be the field of fractions of $Z$, and define the dimension of $A$ over $Z$ as $\dim_{\tilde{Z}}(A\otimes_Z\tilde{Z})$. This dimension is always a square, and its square root is called the PI degree of $A$.

\begin{thm}[See e.g. \cite{BGQGroup,BY,DP}]
Every point in $\centV$ correspond to at least one irreducible representation of $A$. Let $M$ be the PI degree of $A$. Every irreducible representation of $A$ has dimension at most $M$. The points in $\centV$ that correspond to $M$-dimensional irreducible representations form an open subset $U\subset\centV$, called the Azumaya locus of $A$. Moreover, the correspondence is one-to-one on $U$.
\end{thm}

When the quantum parameter $q=\zeta$ is a root of unity, the Kauffman bracket skein algebra $A=\mathring{\skein}_\zeta(\surface)$ satisfy the assumptions above, as shown in \cite{FKLUnicity}. The set of peripheral curves, denoted $\peripheral$, is central for all $q$. Additional central elements are given by (a subalgebra of) the image of the Frobenius homomorphism. These elements generate the center. Moreover, $\mathring{\skein}_\zeta(\surface)$ is finitely generated over its center, and the dimension over center is calculated in \cite{FKLDimension}. See also \cite{BWi,BWii,BWiii}.

In this paper, we determine the center and the dimension over the center for the stated skein algebra $\skein_q(\surface)$, which was announced in \cite{LYSurvey}. We assume that $\surface$ is connected with nonempty boundary. The stated skein algebra for surfaces with empty boundary reduces to the Kauffman bracket skein algebra, so existing results apply.

Suppose $\zeta$ is a primitive $n$-th root of unity. Let $d=\gcd(n,4)$, $N=n/d$, and $\epsilon=\zeta^{N^2}$. \cite{BLFrob} defines the Frobenius homomorphism for the stated skein algebra
\[\Phi_\zeta:\skein_\epsilon(\surface)\to\skein_\zeta(\surface),\]
which is an algebra embedding.

\begin{thm}[{Theorems~\ref{thm-center-root-1} and \ref{thm-dim-z}}]
\begin{enumerate}
\item If $n$ is odd, the center of $\skein_\zeta(\surface)$ is the subalgebra generated by the image of $\Phi_\zeta$, the peripheral curves, and a finite set of central elements $B_\zeta$ (defined in Lemma~\ref{lemma-central-boundary}).
\item For all roots of unity, the center is spanned by elements $\gamma$ such that $\gamma\beta$ is of the form $c\Phi_\zeta(\gamma')$ where $c$ is a polynomial of the peripheral curves, $\gamma'\in\skein_\epsilon(\surface)$ belongs to a subalgebra $X_\zeta$ (defined in Corollary~\ref{cor-Xzeta}), and $\beta$ is a product of elements in $B_\zeta$.
\item The dimension of $\skein_\zeta(\surface)$ over its center is
\[D_\zeta=\begin{cases}
N^{\abs{\bar{\quasi}}-b_2},&d=1,\\
2^{2\lfloor\frac{v-1}{2}\rfloor}N^{\abs{\bar{\quasi}}-b_2},&d=2,\\
2^{2g+2v-2}N^{\abs{\bar{\quasi}}-b_2},&d=4,
\end{cases}\]
where $g$ is the genus of $\surface$, $v\ge 1$ is the number of points of $\marked$ on the boundary $\partial\surclose$, $\abs{\bar{\quasi}}=3(v-\chi(\surclose))+2\abs{\peripheral}$, and $b_2$ is the number of boundary components of $\surclose$ containing an even number of points of $\marked$.
\end{enumerate}
\end{thm}

In \cite{KK}, corresponding results when the order $n$ of $\zeta$ is odd are obtained for a special case ($d=1$ with one point of $\marked$ on each boundary component) and for a quotient of $\skein_q(\surface)$ called the reduced skein algebra.

\section{Punctured bordered surfaces}

\subsection{Basic definitions}

A \term{(punctured bordered) surface} $\surface$ is a surface of the form $\surclose\setminus\marked$, where $\surclose$ is a compact oriented surface, and $\marked$ is a finite set of \term{ideal points} containing at least one point from each boundary component of $\surclose$. For convenience, we assume $\surface$ is connected and has nonempty boundary (so $\marked$ is nonempty as well).

The set of ideal points on the boundary is denoted $\marked_\partial:=\marked\cap\partial\surclose$, while the set of ideal points in the interior, also called punctures, is denoted $\peripheral:=\marked\cap\mathring{\surface}$.

An \term{ideal arc} of $\surface$ is an embedding $a:(0,1)\to\surface$ that extends to an immersion $\bar{a}:[0,1]\to\surclose$ with \term{endpoints} $\bar{a}(0),\bar{a}(1)$ in $\marked_\partial$. Note we do not allow ideal arcs to end on punctures. As usual, we identify an ideal arc with its image. Isotopies of ideal arcs are considered in the class of ideal arcs. If $\bar{a}(0)=\bar{a}(1)$ and $\bar{a}$ bounds a disk in $\surface$, $a$ is called a \term{trivial} ideal arc.

We say the surface $\surface$ is triangulable if it is not a monogon or bigon. A \term{(quasi)triangulation} of $\surface$ is a maximal collection $\quasi$ of nontrivial, pairwise non-isotopic ideal arcs called \term{edges}. Triangulation exists if $\surface$ is not a monogon. The bigon has a unique triangulation consisting of one edge connecting the ideal points. However, the combinatorics of the bigon triangulation is very different from the other surfaces, so we consider the bigon as exceptional. 

For each boundary component of $\surface$, there is an edge of $\quasi$ isotopic to it. Such an edge is called \term{boundary}, and we always assume it is exactly on the boundary of $\surface$. The other edges are called \term{interior}. The collection of boundary edges are denoted $\quasi_\partial$.

The interior edges of a triangulation cut a triangulable surface into a disjoint union of triangles and punctured monogons. Let $\face(\quasi)$ denote these components, which are called the \term{faces} of the triangulation. Since arcs must end on the boundary of $\surclose$, each triangle has three distinct edges.

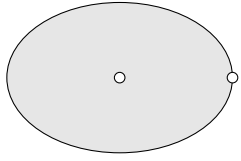
\begin{figure}
\centering
\begin{tikzpicture}
\draw[fill=gray!20] (0,0)circle[x radius=1.5,y radius=1];
\draw[fill=white] (0,0)circle(2pt) (1.5,0)circle(2pt);
\end{tikzpicture}
\caption{A punctured monogon}\label{punctured-monogon}
\end{figure}

\subsection{Matrices associated with a triangulation}

Following \cite{LYSL2}, define the following matrices associated to a triangulation $\quasi$ of a triangulable surface $\surface$.

The edges of a triangle are cyclically ordered. For each triangle $\tau\in\face(\quasi)$ and $a,b\in\quasi$, define 
\begin{equation}
Q_\tau(a,b)=\begin{cases}
1,&a\to b\text{ is clockwise},\\
-1,&a\to b\text{ is counterclockwise},\\
0,&\text{otherwise}.
\end{cases}
\qquad
Q=\sum_{\tau\in\face(\quasi)}Q_\tau.
\end{equation}
The matrix $Q$ is $Q_\quasi$ in \cite{LYSL2}.

Given an edge in $\quasi$, removing a point in the interior produces two \term{half-edges}. For each ideal point $v\in\marked_\partial$, if $a'$ and $b'$ are disjoint half-edges that meet at $v$, define
\begin{equation}
P'_{+,v}(a',b')=\begin{cases}
1,&b'\text{ is counterclockwise to }a',\\
0,&\text{otherwise}.
\end{cases}
\end{equation}
Given two edges $a,b\in\quasi$, isotope them so that they are disjoint. Let
\begin{equation}
P_{+,v}(a,b)=\sum P'_{+,v}(a',b'),\qquad
P_+=\sum_{v\in\marked}P_{+,v},
\end{equation}
where the first sum is over half-edges $a'$ of $a$ and half-edges $b'$ of $b$. The disjoint condition is essential for the definition to make sense when $a=b$.

Some examples are given in Figure~\ref{def-qp-illu}. 

\begin{figure}
\centering
\begin{subfigure}[b]{0.3\linewidth}
\centering
\begin{tikzpicture}[baseline=(C)]
\coordinate (A) at (-30:1.5);
\coordinate (B) at (90:1.5);
\coordinate (C) at (-150:1.5);
\draw[fill=gray!20] (A) to node[above right]{$a$} (B) to (C) to node[below]{$b$} (A);
\node (O) at (0,0){$\tau$};
\draw[fill=white] (A)circle(2pt) (B)circle(2pt) (C)circle(2pt);
\end{tikzpicture}
\subcaption{$Q_\tau(a,b)=1$}
\end{subfigure}
\begin{subfigure}[b]{0.3\linewidth}
\centering
\begin{tikzpicture}[baseline=0cm]
\node[below] (V) at (0,0){$v$};
\fill[gray!20!white] (1.5,0)--(1.5,1.7)--(-1.5,1.7)--(-1.5,0);
\draw (0,0) edge (0:1.5) edge node[above right,at end]{$a'$} (45:1.5) edge node[above left,at end]{$b'$} (135:1.5) edge (180:1.5) edge (90:1.5);
\draw[fill=white] (0,0) circle(2pt);
\end{tikzpicture}
\subcaption{$P'_{+,v}(a',b')=1$}
\end{subfigure}
\caption{Illustrations of the definitions of $Q_\tau$ and $P'_{+,v}$}
\label{def-qp-illu}
\end{figure}
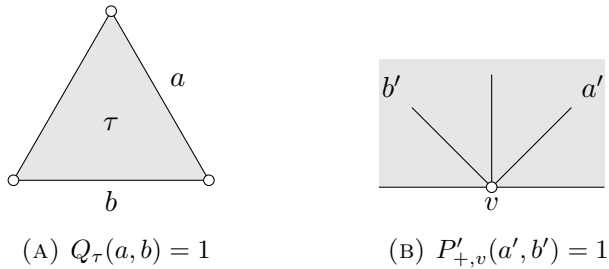

For a surface $\surface$ with a triangulation $\quasi$, there is a relation between the face matrix and the vertex matrix. Define a $\quasi\times\quasi_\partial$ matrix $J$ by
\begin{equation}
J(a,b)=\begin{cases}1,&a=b,\\0,&a\ne b.\end{cases}
\end{equation}

\begin{lemma}[{\cite[Lemma~A.1]{LYSL2}}]\label{lemma-H-inverse}
$P_+(JJ^T-Q)=2I$.
\end{lemma}

Let $\bar{\quasi}=\quasi\sqcup\hat\quasi_\partial$, where $\hat\quasi=\{\hat{e}\mid e\in\quasi_\partial\}$ is a second copy of $\quasi_\partial$. The extended version of $Q$ is a $\bar{\quasi}\times\bar{\quasi}$ block matrix with respect to the above decomposition of $\bar{\quasi}$.
\begin{equation}
\bar{Q}=\begin{pmatrix}Q&-J\\J^T&0\end{pmatrix}.
\end{equation}
The matrix $\bar{Q}$ differs from $\bar{Q}$ in \cite{LYSL2} by the signs on $J$ and $J^T$ because of the choice of generators of the quantum torus defined in Section~\ref{sec-qtr}.

\section{Stated Skein algebra}

\subsection{Stated tangles}

Given a punctured bordered surface $\surface$, a \term{tangle} $\alpha$ over $\surface$ is a compact $1$-dimensional smooth submanifold embedded in $\surface\times(-1,1)$ with a framing (normal vector field) such that
\begin{enumerate}
\item $\partial\alpha\subset\partial\surface\times(-1,1)$;
\item For each boundary edge $b$, the points in $\partial\alpha\cap(b\times(-1,1))$ have distinct heights;
\item The framing at each $x\in\partial\alpha$ is vertical, i.e., tangent to $\{x\}\times(-1,1)$.
\end{enumerate}
A \term{stated tangle} is a tangle $\alpha$ equipped with a \term{state} map $\partial\alpha\to\{+,-\}$. An isotopy of stated tangles is a homotopy through stated tangles. In particular, states and height ordering on a boundary edge are preserved. By convention, the empty set is included as a stated tangle.

When representing tangles by diagrams, height ordering of endpoints on a boundary edge is represented by an arrow, indicating that the heights are increasing as one follow the direction of the arrow. Any tangle can be isotoped to have such a projection. If the arrows on all boundary edges agree with the orientation induced from the surface, then the diagram is called \term{positively ordered}. All diagrams in this paper are positively ordered.

\subsection{Stated skein algebra}
\label{sec-stateS}

Let $R$ be an integral domain with an invertible element $q^{1/2}$. The main example is $R=\cx$ with $q^{1/2}$ a nonzero complex number. The \term{stated skein algebra} $\skein_q(\surface)$ is the $R$-module generated by isotopy classes of stated tangles modulo the following relations.
\begin{align}
&\text{Skein relation:}&
&\TanglePic{1.1}{0.9}{}{}{\tikzmath{\xl=0.1;\xr=\xw-\xl;}
\begin{knot}
\strand[edge] (\xr,\yt)--(\xl,\yd); \strand[edge] (\xl,\yt)--(\xr,\yd);
\end{knot}}
=q\TanglePic{1.1}{0.9}{}{}{
\foreach \x in {0.1,\xw-0.1} \draw[edge] (\x,\yt)..controls (ref)..(\x,\yd);}
+q^{-1}\TanglePic{1.1}{0.9}{}{}{ \tikzmath{\xl=0.1;\xr=\xw-\xl;}
\foreach \y in {\yt,\yd} \draw[edge] (\xl,\y)..controls (ref)..(\xr,\y);}.\label{eq-defrel1}\\
&\text{Trivial loop relation:}&
&\circlediag=(-q^2-q^{-2})\TanglePic{0.9}{0.9}{}{}{}.\\
&\text{State exchange relation:}&
&\twowall{->}{$-$}{$+$}=q^2\twowall{->}{$+$}{$-$}+q^{-1/2}\capnearwall.\label{eq-defrel3}\\
&\text{Returning arc relations:}&
&\capwall{->}{right}{$+$}{$+$}=\capwall{->}{right}{$-$}{$-$}=0,\qquad
\capwall{->}{right}{$+$}{$-$}=q^{-1/2}\TanglePic{0.9}{0.9}{}{w}{}.\label{eq-defrel4}
\end{align}
The product of two stated tangles is defined by stacking $[L_1][L_2]=[i_+(L_1)\cup i_-(L_2)]$ where $i_\pm:M\to M$ are the embeddings given by $i_\pm(x,t)=(x,\frac{t\pm1}{2})$. This extends linearly to a well-defined product.

The height exchange moves
\begin{equation}\label{eq-height-ex}
\crosswall{->}{n}{$\mu$}{$+$}=q^{-\mu}\twowall{->}{$+$}{$\mu$},\qquad
\crosswall{->}{n}{$-$}{$\mu$}=q^{\mu}\twowall{->}{$\mu$}{$-$}
\end{equation}
are consequences of the defining relations. Here we identify the states $\pm$ with $\pm1$. They are also equivalent to the state exchange relation assuming the other defining relations.

\subsection{Exceptional surfaces}

We have assumed throughout that $\surface$ has nonempty boundary. If $\surface$ has empty boundary, tangles can no longer have endpoints, and the boundary relations \eqref{eq-defrel3} and \eqref{eq-defrel4} are vacuous. The stated skein algebra reduces to the Kauffman bracket skein algebra.

If $\surface$ is a monogon, it is easy to see that all tangles can be reduced to the empty tangle using the defining relations. By Theorem~\ref{thm-basis}, the empty tangle is a basis, so the stated skein algebra is simply $R$.

If $\surface$ is a bigon, the stated skein algebra is isomorphic to the quantum coordinate ring $\mathcal{O}_{q^2}(SL_2)$. Additional structures can be defined so that the isomorphism preserves the cobraided Hopf algebra structure. See \cite[Section~3]{CL}. The representation theory is well studied. For results on the center and the PI degree, see \cite[Chapter~III.3]{BGQGroup} and \cite[Appendix]{DL} for $q^2$ an odd root of unity.

\subsection{Basis}

A stated tangle diagram is \term{simple} if there are no crossings, \term{essential} if there are no components homotopic (rel endpoints) to a point or part of a boundary edge, and \term{increasingly stated} if on each boundary edge, the endpoints with the $+$ state are higher than the endpoints with the $-$ state. 

\begin{theorem}[{\cite[Theorem~2.8]{LeTriang}}]\label{thm-basis}
As an $R$-module, $\skein_q(\surface)$ is free. A basis is given by the equivalence classes of increasingly stated, positively ordered, simple essential stated tangle diagrams.
\end{theorem}

A \term{peripheral curve} is a simple closed curve that bounds a disk with one interior ideal point. Peripheral curve are central because they have no intersection with other diagrams up to isotopy, so their heights can be freely isotoped when stacked. Then the corollary below easily follows from Theorem~\ref{thm-basis}.

\begin{corollary}\label{cor-peri}
Let $X_v$ denote the peripheral curve around the interior ideal point $v\in\peripheral$. Then the polynomial algebra
\begin{equation}
\perialg:=R[X_v,v\in\peripheral]
\end{equation}
is an embedded subalgebra of the center of $\skein_q(\surface)$. Therefore, $\skein_q(\surface)$ is an $\perialg$-algebra.

As an $\perialg$-module, $\skein_q(\surface)$ is free with a basis $B$ given by $R$-basis elements of $\skein_q(\surface)$ with no peripheral components.
\end{corollary}

\subsection{Ideal tangle diagrams}
\label{sec-move-left}

In a few places, we need tangle diagrams that end on $\marked_\partial$ instead of on $\partial\surface$, which is away from $\marked_\partial$. We call these \term{ideal tangle diagrams}. In particular, ideal arcs are ideal tangle diagrams.

Given an ideal tangle diagram $\alpha$, it defines a usual tangle diagram $D(\alpha)$ by moving the endpoints slightly in the negative direction of the boundary in a way that does not introduce extra crossings. See Figure~\ref{fig-move-left}. This operation is clearly invertible (up to isotopy).

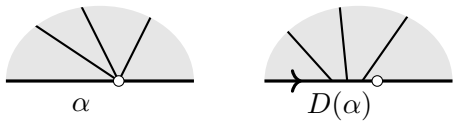
\begin{figure}
\centering
\begin{tikzpicture}[baseline=0cm]
\begin{scope}
\clip (1,0)arc[x radius=1.25,y radius=1,start angle=0,end angle=180];
\fill[gray!20] (-1.5,0) rectangle (1,1);
\draw[edge] (0,0) edge (0.5,1) edge (-0.5,1) edge (-1.5,1);
\end{scope}
\draw[wall] (-1.5,0)--(1,0);
\path (-0.5,0)node[below]{$\vphantom{D}\alpha$};
\draw[fill=white] (0,0)circle(2pt);
\end{tikzpicture}
\qquad
\begin{tikzpicture}[baseline=0cm]
\begin{scope}
\clip (1,0)arc[x radius=1.25,y radius=1,start angle=0,end angle=180];
\fill[gray!20] (-1.5,0) rectangle (1,1);
\draw[edge] (0.5,1)--(-0.2,0) (-0.5,1)--(-0.4,0) (-1.5,1)--(-0.6,0);
\end{scope}
\path[wall,tips,->] (-1.5,0)--(-1,0);
\draw[wall] (-1.5,0)--(1,0);
\path (-0.5,0)node[below]{$D(\alpha)$};
\draw[fill=white] (0,0)circle(2pt);
\end{tikzpicture}
\caption{Tangle diagram $D(\alpha)$}\label{fig-move-left}
\end{figure}

\subsection{Parameterization of basis elements}

Given a triangulation $\quasi$ of $\surface$, the $\perialg$-basis $B$ in Corollary~\ref{cor-peri} can be parameterized by edge colorings. An \term{edge coloring} is a vector $\bar{k}\in\ints^{\bar{\quasi}}$, which is also considered as a map $\bar{\quasi}\to\ints$.
A basis element $\alpha\in B$ represented by a simple tangle diagram with state $s:\partial\alpha\to\{\pm\}$ defines an edge coloring $\bar{k}_\alpha:\bar{\quasi}\to\nats$ by
\begin{equation}\label{eq-color-def}
\begin{aligned}
\bar{k}_\alpha(e)&:=I(\alpha,e),&e&\in\quasi,\\
\bar{k}_\alpha(\hat{e})&:=I(\alpha,e)-\sum_{x\in\alpha\cap e}s(x),& e&\in\quasi_\partial,
\end{aligned}
\end{equation}
where $I$ denotes geometric intersection number. This is defined in \cite[Section~6.3]{LYSL2} as $\bar{\mathbf{n}}_\alpha$, but we want to reserve $n$ for another use. It is easy to see that $\bar{k}_\alpha(\hat{e})$ is twice the number of $-$ skein on $e$. An edge coloring arising this way is called \term{admissible}. Let $\Lambda\subset\nats^{\bar\quasi}$ be the set of admissible edge colorings, and let $\balLambda\subset\ints^{\bar\quasi}$ be the subgroup generated by $\Lambda$. Edge colorings in $\balLambda$ are called \term{balanced}.

\begin{proposition}[{\cite[Lemma~6.2, Proposition~6.3]{LYSL2}}]\label{prob-bal}
$\balLambda$ consists of edge colorings $\bar{k}\in\ints^{\bar\quasi}$ such that
\begin{enumerate}
\item $\bar{k}(a)+\bar{k}(b)+\bar{k}(c)$ is even if $a,b,c\in\quasi$ are edges of a face $\tau\in\face(\quasi)$, and
\item $\bar{k}(a)$ is even if $a$ bounds a punctured monogon or if $a\in\hat\quasi_\partial$.
\end{enumerate}
$\Lambda\subset\nats^{\bar\quasi}\cap\balLambda$ is the submonoid defined by the additional inequalities
\begin{enumerate}
\item $\bar{k}(a)\le\bar{k}(b)+\bar{k}(c)$ if $a,b,c\in\quasi$ are edges of a face $\tau\in\face(\quasi)$, and
\item $\bar{k}(\hat{e})\le 2\bar{k}(e)$ for all $e\in\quasi_\partial$.
\end{enumerate}
An edge coloring $\bar{k}\in\ints^{\bar\quasi}$ is in $\balLambda$ if and only if there exists $\bar{l}\in\Lambda$ such that $\bar{k}-\bar{l}\in(2\ints)^{\bar\quasi}$.
\end{proposition}

The balanced subgroup $\balLambda$ has another useful description. Each edge $a\in\quasi$ defines a basis element $X_a$ of $\skein_q(\surface)$ by assigning $+$ states to both endpoints of the tangle diagram $D(a)$ defined in Section~\ref{sec-move-left}. Each boundary edge $e\in\quasi_\partial$ defines an additional basis element $X_{\hat{e}}$ by a different state assignment shown in Figure~\ref{e-hat-def}. Therefore, we get the corresponding edge colorings $\bar{k}_a:=\bar{k}_{X_a}$, $a\in\bar{\quasi}$. Define a $\bar\quasi\times\bar\quasi$ matrix
\begin{equation}
\bar{K}(a,b)=\bar{k}_a(b).
\end{equation}
Comparing Figure~\ref{e-hat-def} with the definition of $P_+$, we see
\begin{equation}
\bar{K}=\begin{pmatrix}P_+&0\\J^TP_+&2I\end{pmatrix}.
\end{equation}

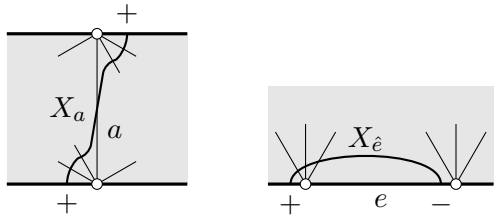
\begin{figure}[h]
\centering
\begin{tikzpicture}[baseline=0cm]
\coordinate (O) at (0,0);
\coordinate (E) at (90:2);
\fill[gray!20] (1.2,0) rectangle (-1.2,2);
\draw[wall] (-1.2,0) -- (1.2,0) (-1.2,2) -- (1.2,2);
\draw (O) edge +(150:0.6) edge +(120:0.6) edge (E) edge +(30:0.6);
\draw (E) edge +(-30:0.6) edge +(-60:0.6) edge +(-150:0.6);
\draw (-0.4,0)node[below]{$+$} (E) ++ (0.4,0)node[above]{$+$};
\draw (0,1) node[left]{$X_a$} (0,0.7) node[right]{$a$};
\draw[edge,rounded corners] (E) ++(0.4,0)
	arc[radius=0.4,start angle=0, end angle=-75]
	--(105:0.4)
	arc[radius=0.4,start angle=105, end angle=180];
\draw[fill=white] (O)circle(2pt) (E)circle(2pt);
\end{tikzpicture}
\qquad
\begin{tikzpicture}[baseline=0cm]
\fill[gray!20] (1.5,0)--(1.5,1.3)--(-1.5,1.3)--(-1.5,0);
\draw[wall] (-1.5,0)--(1.5,0);
\draw (-1,0) edge ++(120:0.8) edge ++(90:0.8) edge ++(60:0.8);
\draw (1,0) edge ++(120:0.8) edge ++(90:0.8) edge ++(60:0.8);
\draw[edge] (-1.2,0)node[below]{$+$}
	..controls (-1.2,0.5) and (0.8,0.5)..(0.8,0)node[below]{$-$};
\draw (0,0)node[below]{$e$} (-0.2,0.6)node{$X_{\hat{e}}$};
\draw[fill=white] (-1,0) circle(2pt) (1,0) circle(2pt);
\end{tikzpicture}
\caption{Basis elements associated to $\bar\quasi$}\label{e-hat-def}
\end{figure}

\begin{lemma}\label{lemma-bal-basis}
The subgroup of balanced vectors is 
a free abelian group with a basis given by the rows of $\bar{K}$, or in other words, $\bar{k}_a,a\in\bar\quasi$.

Consequently, a vector $\bar{k}$ is balanced if and only if it can be written as $\bar{k}=(kP_+,2\hat{k})$ for some $k\in\ints^{\quasi}$ and $\hat{k}\in\ints^{\hat\quasi_\partial}$.
\end{lemma}

\begin{proof}
The generation is a corollary of Theorem~\ref{thm-qtr}. The linear independence follows from the invertiblility of $P_+$ in Lemma~\ref{lemma-H-inverse}.
\end{proof}

\subsection{Filtrations}

The edge colorings define filtrations of the skein algebra. Let $\dego:\ints^{\bar\quasi}\to\ints$ be the homomorphism
\begin{equation}
\dego(\bar{k})=\sum_{e\in\quasi}\bar{k}(e).
\end{equation}
Note the edges in $\hat\quasi_\partial$ are not summed. The filtration $\{F^\circ_d\}_{d\in\nats}$ is given by
\begin{equation}
F^\circ_d=\perialg\text{-span of }\{\alpha\in B\mid\dego(\bar{k}_\alpha)\le d\}.
\end{equation}

For later uses, we also need a refinement of $\{F^\circ_d\}$ that breaks all ties. Let $<$ be a well ordering on $\ints^{\bar\quasi}$ such that $\bar{k}\le\bar{l}$ implies $\dego(\bar{k})\le\dego(\bar{l})$.
Let $\{F_{\bar{k}}\}_{\bar{k}\in\Lambda}$ be the filtration defined by
\begin{equation}
F_{\bar{k}}=\perialg\text{-span of }\{\alpha\in B\mid \bar{k}_\alpha \le \bar{k}\}.
\end{equation}
For an edge coloring $\bar{k}$ with $\dego(\bar{k})=d$, it is easy to see
\begin{equation}
F^\circ_{d-1}\subset \bigcup_{\bar{l}<\bar{k}}F_{\bar{l}},\qquad
F_{\bar{k}}\subset F^\circ_d.
\end{equation}


\subsection{Quantum trace map}\label{sec-qtr}

An advantage of working with surfaces that are not closed is that there is an algebra embedding, called the quantum trace map, from $\skein_q(\surface)$ into a quantum torus $\extY$, which has a much simpler algebra structure. Define the Laurent polynomial algebra
\begin{equation}
\periext:=R[z_v^{\pm1},v\in\peripheral]
\end{equation}
It contains $\perialg$ via the identification $X_v=z_v+z_v^{-1}$. Then $\extY$ is defined as the $\periext$-algebra with the presentation
\begin{equation}
\extY=\periext\langle z_e^{\pm1},e\in\bar{\quasi}\rangle / (z_az_b=q^{\bar{Q}(a,b)}z_bz_a).
\end{equation}
Number the edges $\bar\quasi=\{e_1,\dotsc,e_r\}$ and define the (Weyl-normalized) monomial
\begin{equation}
z^{\bar{k}}=q^{-\frac{1}{2}\sum_{i<j}\bar{Q}(e_i,e_j)\bar{k}(e_i)\bar{k}(e_j)}z_{e_1}^{\bar{k}(e_1)}z_{e_2}^{\bar{k}(e_2)}\dotsm z_{e_r}^{\bar{k}(e_r)}\qquad\text{for }\bar{k}\in\ints^{\bar\quasi}.
\end{equation}
The product of monomials is
\begin{equation}
z^{\bar{k}}z^{\bar{l}}
=q^{\frac{1}{2}\langle\bar{k},\bar{l}\rangle_{\bar{Q}}}z^{\bar{k}+\bar{l}}
=q^{\langle\bar{k},\bar{l}\rangle_{\bar{Q}}}z^{\bar{l}}z^{\bar{k}},
\end{equation}
where
\begin{equation}
\langle\bar{k},\bar{l}\rangle_{\bar{Q}}=\sum_{a,b\in\bar\quasi}\bar{k}(a)\bar{Q}(a,b)\bar{l}(b)=\bar{k}^T\bar{Q}\bar{l}
\end{equation}
is the skew-symmetric bilinear form associated to $\bar{Q}$. 

The product formula shows that $\extY$ has a $\ints^{\bar\quasi}$-grading
\begin{equation}
\extY=\bigoplus_{\bar{k}\in\ints^{\bar\quasi}}\periext z^{\bar{k}}.
\end{equation}
The grading can be reduced to filtrations $\{F^\circ_d\}$ and $\{F_{\bar{k}}\}$, defined similarly to the skein algebra case.

\begin{theorem}[{\cite[Theorems~6.5, 7.1 and Lemma~7.5]{LYSL2}}]\label{thm-qtr}
Suppose $\surface$ is a triangulable surface. There is an $\perialg$-algebra embedding
\begin{equation}
\phi:\skein_q(\surface)\to\extY
\end{equation}
such that
\begin{enumerate}
\item Suppose $\alpha\in B$ is a basis element with edge coloring $\bar{k}_\alpha$. Let $d=\dego(\bar{k}_\alpha)$. Then there exists a nonzero $c(\alpha)\in\periext$ such that
\begin{equation}
\phi(\alpha)=c(\alpha)z^{\bar{k}_\alpha}\mod F^\circ_{d-1}.
\end{equation}
\item If $\alpha=X_a$ is associated to $a\in\bar\quasi$, then $\phi(X_a)$ is a monomial. In other words, the equality above holds without ``$\bmod F^\circ_{d-1}$".
\item The image of $\phi$ is contained in the subalgebra generated by $\phi(X_a)^{\pm1},a\in\quasi$ and $z_v^{\pm1},v\in\peripheral$.
\end{enumerate}
\end{theorem}

Define the \term{degree} and \term{leading term} of a nonzero element $\alpha\in\skein_q(\surface)$ with respect to the filtration $\{F_{\bar{k}}\}$. In other words, $\deg(\alpha)$ is the edge coloring $\bar{k}\in\ints^{\bar\quasi}$ such that
\begin{equation}
\alpha=c_0\alpha_0\mod\bigcup_{\bar{l}<\bar{k}}F_{\bar{l}},
\end{equation}
where $c_0\in\perialg$ is nonzero, and $\alpha_0\in B$ is the basis element with edge coloring $\bar{k}$. Then the leading term of $\alpha$ is defined as $\lead(\alpha)=c_0\alpha_0$.

By Theorem~\ref{thm-qtr}(1), the degree of $\alpha$ can be obtained from the degree of $\phi(\alpha)\in\extY$, which is much easier since there is actually a grading. The following result is an easy corollary.

\begin{corollary}\label{cor-deg-lt}
Let $\alpha,\beta\in\skein_q(\surface)\setminus\{0\}$. Then
\begin{enumerate}
\item $\deg(\alpha\beta)=\deg\alpha+\deg\beta$. 
\item $\lead(\alpha\beta)=q^{\langle\deg\alpha,\deg\beta\rangle_{\bar{Q}}}\lead(\beta\alpha)$.
\end{enumerate}
\end{corollary}


\subsection{The Frobenius homomorphism}

Suppose $\zeta$ is a root of unity. Let $n=\order(\zeta)$, $d=\gcd(n,4)$, $N=n/d=\order(\zeta^4)$, and $\epsilon=\zeta^{N^2}$. Strictly speaking, we should also define $\epsilon^{1/2}=(\zeta^{1/2})^{N^2}$, which will be implied throughout the rest of the paper. 

\begin{theorem}[\cite{BLFrob}]\label{thm-thread}
There exists an algebra map
\begin{equation}
\Phi_\zeta:\skein_\epsilon(\surface)\to\skein_\zeta(\surface),
\end{equation}
called the \term{Frobenius homomorphism}, satisfying the following properties.
\begin{enumerate}
\item If $\alpha$ is an arc, then $\Phi_\zeta(\alpha)$ is represented by the tangle with $N$ components, each of which is a copy of $\alpha$ shifted slightly in the direction of the framing. 
\item More generally, for a tangle $\alpha$, $\Phi_\zeta(\alpha)$ is a linear combination of tangles obtained by replacing each arc component of $\alpha$ with exactly $N$ parallel copies and each closed component of $\alpha$ with $\le N$ parallel copies. The coefficient of the term where every component becomes $N$ copies is $1$. Consequently,
\begin{equation}
\deg\Phi_\zeta(\alpha)=N\deg\alpha.
\end{equation}
\item ((Skew-)transparency.) Let $\alpha$ be a tangle disjoint from isotopic tangles $\beta_0,\beta_1$. If $\alpha\cup\beta_0$ and $\alpha\cup\beta_1$ have diagrams that differ by a crossing change, then
\begin{equation}\label{eq-skew-trans}
\Phi_\zeta(\alpha)\cup\beta_0=\zeta^{2N}\Phi_\zeta(\alpha)\cup\beta_1.
\end{equation}
\end{enumerate}
\end{theorem}

Combining (1) with height exchange moves, it is easy to see that if $\alpha$ is a simple arc diagram, then $\Phi_\zeta(\alpha)$ and $\alpha^N$ differ by a power of $\zeta$.

\section{Characterization of the center}

\subsection{Center at a generic $q$}

As a warm-up, we show that at generic $q$, the center is the obvious one.

\begin{theorem}\label{thm-center-generic}
If $q$ is not a root of unity, then the center of $\skein_q(\surface)$ is $\perialg$.
\end{theorem}

\begin{proof}
Suppose $\alpha$ is central. Let $\bar{k}=\deg \alpha$. Since $\lead(\alpha\beta)=\lead(\beta\alpha)$ for any $\beta\in\skein_q(\surface)$, $\bar{k}_\beta^T\bar{Q}\bar{k}=0$ by Lemma~\ref{cor-deg-lt}. By choosing $\beta=X_a$ for $a\in\bar\quasi$, we get
\begin{equation}
\bar{K}\bar{Q}\bar{k}=0.
\end{equation}
Using block matrix notations,
\begin{align}
\begin{pmatrix}I&0\\-J^T&I\end{pmatrix}\bar{K}\bar{Q}
&=\begin{pmatrix}P_+&0\\0&2I\end{pmatrix}\begin{pmatrix}Q&-J\\J^T&0\end{pmatrix}
=\begin{pmatrix}P_+Q&-P_+J\\2J^T&0\end{pmatrix}\\
&=\begin{pmatrix}P_+JJ^T-2I&-P_+J\\2J^T&0\end{pmatrix}
\end{align}
where Lemma~\ref{lemma-H-inverse} was used in the last step. We can decompose $\bar{k}=(k,2\hat{k})$ according to $\bar\quasi=\quasi\sqcup\hat\quasi_\partial$, where the factor of $2$ comes from the admissible condition. Then we have
\begin{equation}\label{eq-center-generic}
\begin{aligned}
P_+Jk_\partial-2k-2P_+J\hat{k}&=0,\\
2k_\partial&=0.
\end{aligned}
\end{equation}
Here, $k_\partial:=J^Tk$ consists of the components corresponding to boundary edges. Multiplying the first equation with $J^T$ and using $k_\partial=0$, we get $J^TP_+J\hat{k}=0$.

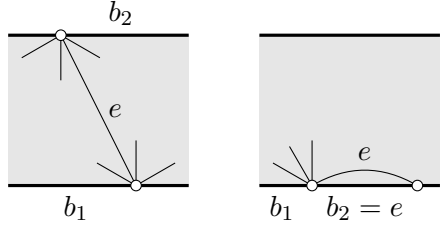
\begin{figure}
\centering
\begin{tikzpicture}[baseline=0cm]
\coordinate (O) at (0.5,0);
\coordinate (E) at (-0.5,2);
\fill[gray!20] (1.2,0) rectangle (-1.2,2);
\draw[wall] (-1.2,0) -- (1.2,0) (-1.2,2) -- (1.2,2);
\draw (O) edge +(150:0.6) edge +(90:0.6) edge +(30:0.6);
\draw (O) -- node[midway,right]{$e$} (E);
\draw (E) edge +(-30:0.6) edge +(-90:0.6) edge +(-150:0.6);
\path (-0.3,0)node[below]{$b_1$} (0.3,2)node[above]{$b_2$};
\draw[fill=white] (O)circle(2pt) (E)circle(2pt);
\end{tikzpicture}
\qquad
\begin{tikzpicture}[baseline=0cm]
\coordinate (O) at (-0.5,0);
\fill[gray!20] (1.2,0) rectangle (-1.2,2);
\draw[wall] (-1.2,0) -- (1.2,0) (-1.2,2) -- (1.2,2);
\draw (O) edge +(150:0.6) edge +(120:0.6) edge +(90:0.6);
\draw (O) to[bend left] node[midway,above]{$e$} (0.9,0);
\path (-0.9,0)node[below]{$b_1$} (0.2,0)node[below]{$b_2=e$};
\draw[fill=white] (O)circle(2pt) (0.9,0)circle(2pt);
\end{tikzpicture}
\caption{Edges $b_1,b_2$}\label{fig-matching}
\end{figure}

Consider the product
\begin{equation}\label{eq-PJk}
(P_+J\hat{k})(e)=\sum_{a\in\quasi}\sum_{b\in\quasi_\partial}P_+(e,a)J(a,b)\hat{k}(b)=\sum_{b\in\quasi_\partial}P_+(e,b)\hat{k}(b)=\hat{k}(b_1)+\hat{k}(b_2),
\end{equation}
where $b_1,b_2\in\quasi_\partial$ are the boundary edges counterclockwise to $e$. See Figure~\ref{fig-matching}. Multiplying $J^T$ simply restricts $e\in\quasi_\partial$. Thus, $J^TP_+J\hat{k}=0$ implies that $\hat{k}(b_1)+\hat{k}(b_2)=0$ if $b_1$ and $b_2$ are adjacent boundary edges. Since $\bar{k}_\alpha$ is admissible, the components are nonnegative. Thus, admissible solutions have $\hat{k}=0$, which implies the degree is trivial. The only elements with trivial degree are polynomials of peripheral curves.
\end{proof}

\subsection{Gradings of the stated skein algebra}

To describe the center when $q=\zeta$ is a root of unity, it is convenient to introduce gradings.

Each tangle diagram defines a homology class in $H=H_1(\surclose,\partial\surface;\ints/2)$. Note we use the compact surface $\surclose$ but only the punctured boundary $\partial\surface=\partial\surclose\setminus\marked_\partial$. This homology class is preserved by all defining relations. Stacking diagrams results in the sum of homology classes. Thus, the stated skein algebra is graded by $H$. This is called the \term{homology class grading}.

There is an equivalent form of the homology class grading. The operation $D$ defined in Section~\ref{sec-move-left} induces an isomorphism
\begin{equation}
D_\ast:H^\ast\to H,\qquad
H^\ast:=H_1(\surclose,\marked_\partial;\ints/2).
\end{equation}
Thus, any tangle diagram $\alpha$ also defines an element $D_\ast^{-1}(\alpha)\in H^\ast$. The notation $H^\ast$ is intended to indicate that it is the dual of $H$ with respect to the $\bmod2$ intersection pairing
\begin{equation}
i_2:H\otimes H^\ast\to\ints/2.
\end{equation}

There is also a $\ints$-grading for each boundary edge $e$. Given a stated diagram $\alpha$, let $\delta_e(\alpha)$ be sum of states of $\alpha$ on $e$, where $\pm$ are identified with $\pm1$. This is compatible with the defining relations and stacking. Such gradings are called \term{boundary gradings}.

The homology class grading and boundary gradings are clearly compatible, so together they form an $H\times\ints^{\quasi_\partial}$ grading. They are not independent, as a boundary grading $\bmod2$ reduces to an intersection number, which is homological.

The gradings can also be determined from edge colorings. Let $\alpha$ be a diagram with edge coloring $\bar{k}_\alpha$. Note each edge $e\in\quasi$ represents a homology class in $H^\ast$. From the definition \eqref{eq-color-def}, we see
\begin{equation}
\begin{aligned}
i_2(\alpha,e)&\equiv\bar{k}_\alpha(e)\bmod2 ,&e&\in\quasi,\\
\delta_e(\alpha)&=\bar{k}_\alpha(e)-\bar{k}_\alpha(\hat{e}),& e&\in\quasi_\partial.
\end{aligned}
\end{equation}
The first equation determines the homology class of $\alpha$ since the edges of a triangulation span the homology group $H^\ast$ and the intersection pairing $i_2$ is non-degenerate.

We define two subalgebras using these gradings. Suppose $\alpha$ is a tangle diagram on the surface $\surface$. $\alpha$ is \term{matching} if all boundary gradings $\delta_e(\alpha)$ (or intersection numbers with the boundary edges) have the same parity. $\alpha$ is \term{even} if the homology class grading is zero, and all boundary gradings are divisible by $4$.
The matching subalgebra $\matchedS(\surface)$ and the even subalgebra $\evenS(\surface)$ are spanned by their corresponding types of diagrams.

\subsection{Commutation relations at a root of unity}

Recall that $\zeta$ denotes a root of unity, $n=\order(\zeta)$, $d=\gcd(n,4)$, $N=n/d=\order(\zeta^4)$, and $\epsilon=\zeta^{N^2}$.

\begin{lemma}\label{lemma-Frob-comm}
Suppose $\alpha\in\skein_\epsilon(\surface)$ and $\beta\in\skein_\zeta(\surface)$ are given by tangle diagrams. Then
\begin{equation}\label{eq-Frob-comm}
\Phi_\zeta(\alpha)\beta=\zeta^{c(\alpha,\beta)N}\beta\Phi_\zeta(\alpha),\qquad
c(\alpha,\beta)=2i_2(\alpha,D^{-1}_\ast(\beta))-\sum_{e\in\quasi_\partial}\delta_e(\alpha)\delta_e(\beta).
\end{equation}
\end{lemma}

Note $c(\alpha,\beta)$ is only well-defined $\bmod4$, but this is enough to make sense of the coefficient in \eqref{eq-Frob-comm}.

\begin{proof}
For each boundary component of $\surclose$, choose a small regular neighborhood. Isotope the diagrams so that
\begin{enumerate}
\item $\alpha$ and $\beta$ do not intersect in the neighborhoods of boundary components.
\item As one follows the positive direction of each boundary edge, the endpoints of $\alpha$ are all before the endpoints of $\beta$.
\end{enumerate}
This is illustrated in Figure~\ref{fig-frob-std}, where one boundary edge is shown, and the dashed line indicates the choice of the neighborhood. Note the movement of $\beta$ into $D^{-1}(\beta)$ does not pass through $\alpha$. Thus, $\abs{\alpha\cap\beta}\equiv i_2(\alpha,D^{-1}_\ast(\beta))\pmod2$.

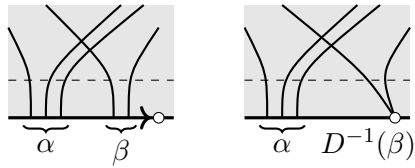
\begin{figure}
\centering
\begin{tikzpicture}[baseline=0cm]
\fill[gray!20] (0,0) rectangle (2.2,1.5);
\draw[dashed] (0,0.5) -- (2.2,0.5);
\draw[wall] (0,0) -- (2.2,0);
\path[wall,tips,->] (0,0) -- (1.9,0);
\begin{scope}[edge]
\draw (1.4,0) .. controls +(0,0.6) .. (0.5,1.5);
\draw (1.6,0) .. controls +(0,0.6) .. (2,1.2);
\draw (0.3,0) .. controls +(0,0.6) .. (0,1.2);
\draw (0.5,0) .. controls +(0,0.6) .. (1.5,1.5);
\draw (0.7,0) .. controls +(0,0.6) .. (1.8,1.5);
\end{scope}
\begin{scope}[thick,decoration=brace]
\draw[decorate] (0.8,-0.1)--(0.2,-0.1);
\draw (0.5,-0.15)node[below]{$\alpha$};
\draw[decorate] (1.7,-0.1)--(1.3,-0.1);
\draw (1.5,-0.15)node[below]{$\beta$};
\end{scope}
\draw[fill=white] (2,0)circle(2pt);
\end{tikzpicture}
\qquad
\begin{tikzpicture}[baseline=0cm]
\fill[gray!20] (0,0) rectangle (2.2,1.5);
\draw[dashed] (0,0.5) -- (2.2,0.5);
\draw[wall] (0,0) -- (2.2,0);
\begin{scope}[edge]
\draw (2,0) .. controls +(-0.4,0.6) .. (0.5,1.5);
\draw (2,0) .. controls +(-0.2,0.6) .. (2.2,1.2);
\draw (0.3,0) .. controls +(0,0.6) .. (0,1.2);
\draw (0.5,0) .. controls +(0,0.6) .. (1.5,1.5);
\draw (0.7,0) .. controls +(0,0.6) .. (1.8,1.5);
\end{scope}
\begin{scope}[thick,decoration=brace]
\draw[decorate] (0.8,-0.1)--(0.2,-0.1);
\draw (0.5,-0.15)node[below]{$\alpha$};
\draw (2,0)node[anchor=45]{$D^{-1}(\beta)$};
\end{scope}
\draw[fill=white] (2,0)circle(2pt);
\end{tikzpicture}
\caption{The standard positions of $\alpha,\beta$ and $D^{-1}(\beta)$}\label{fig-frob-std}
\end{figure}

\begin{figure}
\centering
\begin{tikzpicture}[baseline=1cm]
\fill[gray!20] (0,0) rectangle (2,2);
\begin{scope}[edge]
\draw (0.4,0) .. controls +(1.2,1) .. (0.5,2);
\draw (0.6,0) .. controls +(1.2,1) .. (2,1.5);
\begin{scope}[knot gap=6,background color=gray!20]
\clip (0,0) rectangle (2,2);
\draw[knot=black] (1.3,0) .. controls +(-1.2,1) .. (0,1.5);
\draw[knot=black] (1.7,0) .. controls +(-1.2,1) .. (1.8,2);
\draw[knot=black] (1.5,0) .. controls +(-1.2,1) .. (1.5,2);
\end{scope}
\end{scope}
\draw[dashed] (0,1) -- (2,1);
\draw[wall,->] (0,0) -- (2,0);
\begin{scope}[thick,decoration=brace]
\draw[decorate] (1.8,-0.1)--(1.2,-0.1);
\draw (1.5,-0.15)node[below]{$\Phi_\zeta(\alpha)$};
\draw[decorate] (0.7,-0.1)--(0.3,-0.1);
\draw (0.5,-0.15)node[below]{$\beta$};
\end{scope}
\end{tikzpicture}
\qquad
\begin{tikzpicture}[baseline=0.5cm]
\fill[gray!20] (0,0) rectangle (2,1.5);
\begin{scope}[edge]
\draw (0.3,0) .. controls +(0,0.6) .. (0,1.2);
\draw (0.5,0) .. controls +(0,0.6) .. (1.5,1.5);
\draw (0.7,0) .. controls +(0,0.6) .. (1.8,1.5);
\begin{scope}[knot gap=6,background color=gray!20]
\clip (0,0) rectangle (2,1.5);
\draw[knot=black] (1.4,0) .. controls +(0,0.6) .. (0.5,1.5);
\end{scope}
\draw (1.6,0) .. controls +(0,0.6) .. (2,1.2);
\end{scope}
\draw[dashed] (0,0.5) -- (2,0.5);
\draw[wall,->] (0,0) -- (2,0);
\begin{scope}[thick,decoration=brace]
\draw[decorate] (0.8,-0.1)--(0.2,-0.1);
\draw (0.5,-0.15)node[below]{$\Phi_\zeta(\alpha)$};
\draw[decorate] (1.7,-0.1)--(1.3,-0.1);
\draw (1.5,-0.15)node[below]{$\beta$};
\end{scope}
\end{tikzpicture}

\caption{The schematic diagrams $\Phi_\zeta(\alpha)\beta$ and $\beta\Phi_\zeta(\alpha)$}\label{fig-frob-comm}
\end{figure}

Looking at Figure~\ref{fig-frob-comm}, the process of turning $\Phi_\zeta(\alpha)\beta$ into $\beta\Phi_\zeta(\alpha)$ involves two parts: crossing changes corresponding to $\alpha\cap\beta$ and height exchanges near the boundary. Each crossing change involves a factor of $\zeta^{2N}$ by skew-transparency \eqref{eq-skew-trans}. For height exchanges, note in $\Phi_\zeta(\alpha)$, arcs ending on the boundary comes in $N$ parallel copies. By applying the height exchange moves \eqref{eq-height-ex} repeatedly, we get
\begin{equation}\label{eq-heightex-N}
\crosswall{->}{n}{$N\mu$}{$+$}=\zeta^{-\mu N}\twowall{->}{$+$}{$N\mu$},\qquad
\crosswall{->}{n}{$-$}{$N\mu$}=\zeta^{\mu N}\twowall{->}{$N\mu$}{$-$},
\end{equation}
where $N\mu$ indicates $N$ parallel strands, all with the state $\mu$. We can combine the relations above into a single one
\begin{equation}\label{eq-Frob-ex}
\crosswall{->}{n}{$N\mu$}{$\nu$}=\zeta^{-\mu\nu N}\twowall{->}{$\nu$}{$N\mu$}.
\end{equation}
Here the $\nu=-$ case is obtained from the second equation of \eqref{eq-heightex-N} and skew-transparency. Applying this to $\Phi_\zeta(\alpha)\beta$ in the neighborhoods of boundary edges, we get the remaining factors.
\end{proof}

\begin{corollary}\label{cor-Xzeta}
Let
\begin{equation}
X_\zeta=\begin{cases}
\skein_\epsilon(\surface),&d=1,\\
\matchedS_\epsilon(\surface),&d=2,\\
\evenS_\epsilon(\surface),&d=4.
\end{cases}
\end{equation}
Then $\Phi_\zeta(X_\zeta)\subset\skein_\zeta(\surface)$ is central.
\end{corollary}

\begin{proof}
Suppose $\alpha\in X_\zeta$ and $\beta\in\skein_\zeta(\surface)$ are represented by diagrams.

Case 1 ($d=1$). Since $n=N$, the coefficient in \eqref{eq-Frob-comm} is always $1$. Therefore, $\Phi_\zeta(\alpha)$ commutes with $\beta$ without extra assumptions. (See also \cite[Theorem~1.2]{KQ})

Case 2 ($d=2$). Since $n=2N$, we only need to determine $c(\alpha,\beta)\bmod2$. Thus, the $i_2$ term in \eqref{eq-Frob-comm} has no effect. In the other term, all $\delta_e(\alpha)$ are the same $\bmod2$ because of the matching condition, so it can be factored out of the sum. On the other hand, for any diagram $\beta$, the $\bmod2$ sum of boundary gradings is zero since it agrees with the $\bmod2$ count of endpoints. Thus, $c(\alpha,\beta)$ is even when $\alpha$ is matched, which means $\Phi_\zeta(\alpha)$ commutes with $\beta$.

Case 3 ($d=4$). The $i_2$ term in \eqref{eq-Frob-comm} vanishes since an even $\alpha$ has trivial $\bmod2$ homology. In addition, $\delta_e(\alpha)$ is divisible by 4 by definition. Hence, $c(\alpha,\beta)$ is divisible by $4$, which means $\Phi_\zeta(\alpha)$ commutes with $\beta$.
\end{proof}

\begin{lemma}\label{lemma-central-boundary}
Suppose $e_1,\dotsc,e_r$ are the boundary edges on a boundary component $C$ with $r$ ideal points, ordered consecutively with an arbitrary starting edge.
\begin{enumerate}
\item If $r$ is even, the element $X_{\hat{e}_1}^kX_{\hat{e}_2}^{n-k}\dotsm X_{\hat{e}_r}^{n-k}$ is central for $0\le k\le n$.
\item If $r$ is odd, the element $X_{e_1}^NX_{e_2}^N\dotsm X_{e_r}^N=\Phi_\zeta(X_{e_1}X_{e_2}\dotsm X_{e_r})$ is central.
\end{enumerate}
\end{lemma}

If $d=1,2$, let $B_\zeta$ be the set of the first type of elements. If $d=4$, let $B_\zeta$ be the set of both types of elements. The distinction comes from the proof of Lemma~\ref{lemma-lt-cancel}.

\begin{remark}\label{rem-arc-power}
Since we use positively ordered diagrams, powers of arcs such as $X_{\hat{e}_1}^k$ is not represented by the simple diagram with copies of the arc. However, the difference is a power of $q=\zeta$ using height exchange moves. We will ignore the difference because it does not affect centrality.
\end{remark}

\begin{proof}
Let $\alpha\in\skein_\zeta(\surface)$ be an arbitrary diagram. We want to show the elements above commute with $\alpha$.

Consider an element $\beta$ of the first type. By \cite[Lemma~4.5]{LYSL2}, for a boundary edge $e\in\quasi_\partial$
\begin{equation}
X_{\hat{e}}\alpha=\zeta^{\delta_e(\alpha)+\delta_{e'}(\alpha)}\alpha X_{\hat{e}}.
\end{equation}
Here $e'$ is the edge counterclockwise to $e$, or equivalently, the edge other than $e$ with an endpoint of $X_{\hat{e}}$. Therefore,
\begin{align}
\beta\alpha&=(X_{\hat{e}_1}^kX_{\hat{e}_2}^{n-k}\dotsm X_{\hat{e}_r}^{n-k})\alpha\\
&=(\zeta^{\delta_{e_r}(\alpha)+\delta_{e'_r}(\alpha)})^{n-k}X_{\hat{e}_1}^kX_{\hat{e}_2}^{n-k}\dotsm \alpha X_{\hat{e}_r}^{n-k}=\dotsb\\
&=
(\zeta^{\delta_{e_1}(\alpha)+\delta_{e'_1}(\alpha)})^{k}
\dotsm
(\zeta^{\delta_{e_r}(\alpha)+\delta_{e'_r}(\alpha)})^{n-k}
\alpha(X_{\hat{e}_1}^kX_{\hat{e}_2}^{n-k}\dotsm X_{\hat{e}_r}^{n-k}).
\end{align}
Each $\delta_{e_i}$ appears twice with multiplicities $k$ and $n-k$, one of which is as the prime of another edge. Thus, the coefficient in the last line simplifies to $\zeta^{n\sum\delta_{e_i}(\alpha)}=1$, which means $\beta$ commutes with $\alpha$.

Now consider an element $\beta=\Phi_\zeta(\beta')$ of the second type, where $\beta'=X_{e_1}X_{e_2}\dotsm X_{e_r}\in\skein_\epsilon(\surface)$. By Lemma~\ref{lemma-Frob-comm}, we need to show $c(\beta',\alpha)\equiv0\pmod4$. First note that $\beta'$ is homologous to the boundary component $C$. Thus, $i_2(\beta',D_\ast^{-1}(\alpha))$ is simply the number of endpoints of $\alpha$ on $C$. Next consider the sum over boundary edges. Clearly $\delta_e(\beta')=2$ if $e$ is on $C$ and zero otherwise, so we have
\begin{equation}
c(\beta',\alpha)=2\abs{\alpha\cap C}-2\sum_{e\subset C}\delta_e(\alpha),
\end{equation}
which is 4 times the number of $-$ states on $C$. Thus, $c(\beta',\alpha)$ is a multiple of $4$, so $\beta$ commutes with $\alpha$.
\end{proof}

\subsection{Center at a root of unity}

We are finally ready to determine the center at a root of unity.

\begin{theorem}\label{thm-center-root-1}
Suppose $\surface$ is connected, triangulable, and have nonempty boundary. The center of $\skein_q(\surface)$ at $q=\zeta$ is
\begin{equation}
Z_\zeta=\Phi_\zeta(X_\zeta)[\peripheral][B_\zeta^{-1}]\cap\skein_\zeta(\surface),
\end{equation}
where $\peripheral$ represents the set of peripheral curves, and $B_\zeta$ is the set of central elements defined in Lemma~\ref{lemma-central-boundary}. In other words, the center $Z_\zeta$ is spanned by elements $\gamma$ such that $\gamma\beta$ is of the form $c\Phi_\zeta(\gamma')$ where $c\in\perialg$ and $\beta$ is a product of elements in $B_\zeta$.

If the order of $\zeta$ is odd ($d=1$), then the inverse on $B_\zeta$ can be removed. In other words, the center is the subalgebra of $\skein_\zeta(\surface)$ generated by $\Phi_\zeta(\skein_1(\surface))$, the peripheral curves $\peripheral$, and $B_\zeta$.
\end{theorem}

\begin{proof}
$Z_\zeta$ is central by the previous lemmas. Suppose $\alpha\in \skein_\zeta(\surface)$ is central. By Lemma~\ref{lemma-lt-cancel}, there exists an element $\gamma\in Z_\zeta$ such that $\lead(\alpha)=\lead(\gamma)$. Then the element $\alpha-\gamma$ is central and has lower degree. Since the degree is defined using a well ordering on edge colorings, after repeating the process finitely many times, we can write $\alpha$ as a sum of elements in $Z_\zeta$.
\end{proof}

\begin{lemma}\label{lemma-lt-cancel}
Suppose $\alpha\in\skein_\zeta(\surface)$ is central. There exists an element $\gamma\in Z_\zeta$ such that $\lead(\alpha)=\lead(\gamma)$
\end{lemma}

\begin{proof}
Using the same notations as Theorem~\ref{thm-center-generic}, let $\bar{k}=(k,2\hat{k})$ be the degree of $\alpha$, and $k_\partial:=J^Tk$. The same argument leads to the $\bmod n$ version of \eqref{eq-center-generic}.
\begin{equation}\label{eqn-center}
\begin{split}
P_+Jk_\partial-2k-2P_+J\hat{k}&\equiv0\pmod{n},\\
2k_\partial&\equiv0\pmod{n}.
\end{split}
\end{equation}

\textbf{Case 1} ($d=1$.) Since $n=N$ is odd in this case, the relations are reduced to
\begin{equation}\label{eqn-center-1}
\begin{split}
k+P_+J\hat{k}&\equiv0\pmod{N},\\
k_\partial&\equiv0\pmod{N}.
\end{split}
\end{equation}
Just like the proof of Theorem~\ref{thm-center-generic}, by multiplying the first relation by $J^T$, we can show that if $a$ and $b$ are adjacent boundary edges, then $\hat{k}(a)+\hat{k}(b)$ is divisible by $N$. If there are $r$ ideal points, then by going around the boundary, $\hat{k}(a)\equiv(-1)^r\hat{k}(a)\pmod{N}$. Since $N$ is odd, nontrivial $\bmod{N}$ solutions exist only when $r$ is even. Let
\begin{equation}
\beta=\prod_{e\in\quasi_\partial}X_{\hat{e}}^{-\hat{k}(e)\bmod{n}}.
\end{equation}
Here and in the rest of the proof, $\bmod{n}$ represents the operation that takes the remainder of the division in the range $0,1,\dots,n-1$ if an integer is expected. $\beta$ is a product of elements in $B_\zeta$ if ordered correctly. $\alpha\beta$ is still central since $\beta$ is. Its degree $\bar{k}'=(k',2\hat{k}')$ satisfies the same relations, but now $\hat{k}'\equiv0\pmod{N}$, which implies $k'\equiv0\pmod{N}$ as well. Let $\gamma'\in\skein_\epsilon(\surface)$ be the basis element without peripheral components corresponding to the coloring $\bar{k}'/N$. Then $\alpha\beta$ have the same degree as $\Phi_\zeta(\gamma')$, so $\lead(\alpha\beta)=\lead(c\Phi_\zeta(\gamma'))$ for some $c\in\perialg$.

By \cite[Lemma~4.5]{LYSL2}, when a diagram $\alpha_0$ is multiplied by $X_{\hat{e}}$, assumed to be disjoint from $\alpha_0$ by an isotopy, the result is a scalar multiple of the diagram $\alpha\cup X_{\hat{e}}$. Thus, if $\alpha_0$ is the diagram for $\lead(\alpha)$, then the diagram of $\lead(\alpha\beta)$ is a simple diagram $\alpha_0\cup\beta$, which is also the diagram for $\lead(\Phi_\zeta(\gamma'))$ by construction. By Theorem~\ref{thm-thread}(2), the arc components in $\Phi_\zeta(\gamma')$ is the same in every term. Thus, $\Phi_\zeta(\gamma')$ contains a factor of $\beta$. Let $c\Phi_\zeta(\gamma')=\gamma\beta$. Then $\gamma\in Z_\zeta$ and $\lead(\gamma)=\lead(\alpha)$.

\textbf{Case 1, Stronger version}. Consider the diagram of $\lead(\alpha)$, which has edge coloring $\bar{k}$. Suppose two edges $a_1,a_2$ meet at an ideal point and form a corner, and $e$ is the boundary edge clockwise to both of them. First assume the corner is part of a triangle as in Figure~\ref{fig-pos-gen}. Since $k\equiv -P_+J\hat{k}\pmod{N}$, from \eqref{eq-PJk}, we get
\begin{gather}
k(a_1)\equiv -\hat{k}(e')-\hat{k}(e_1),\qquad
k(a_2)\equiv -\hat{k}(e')-\hat{k}(e_2),\qquad
k(a)\equiv -\hat{k}(e_1)-\hat{k}(e_2)\pmod{N}.\\
k(a_1)+k(a_2)-k(a)\equiv -2\hat{k}(e')\equiv 2\hat{k}(e)\pmod{N}.
\end{gather}
It is well known that $k(a_1)+k(a_2)-k(a)$ is twice the number of corner arcs between $a_1$ and $a_2$. Cancelling the $2$ from the last equation, we see that the number of corner arcs is at least $(\hat{k}(e)\bmod{N})$. A similar calculation can be done when $a_1=a_2$ bounds a punctured monogon with the same result. These corner arcs connect to form at least $(\hat{k}(e)\bmod{N})$ copies of $D(e)$.

\begin{figure}
\centering
\begin{tikzpicture}[baseline=0cm]
\fill[gray!20] (-60:0.8) -- ++(0:0.8)coordinate(A) -- ++(0:0.8)
	-- ++(60:1.6) -- ++(120:0.8)coordinate(B) -- ++(120:0.8)
	-- ++(180:1.6) -- ++(-120:0.8) coordinate(C) -- ++(-120:0.8) -- cycle;
\draw[wall] (A) ++(-0.8,0) -- +(1.6,0)
	(B) ++(-60:0.8) -- +(120:1.6) (C) ++(60:0.8) -- +(-120:1.6);
\draw (A) -- node[midway,right]{$a_2$} (B)
	-- node[midway,above]{$a$} (C)
	-- node[midway,right]{$a_1$} cycle;
\path[inner sep=2pt] (A) +(-0.6,0)node[below]{$e'$}
	+(0.6,0)node[below]{$e\vphantom{e'}$};
\path (B) +(-60:0.4)node[right]{$e_2$} (C) +(60:0.4)node[left]{$e_1$};
\draw[edge] (A) ++(0.4,0)node[inner sep=2pt,above right]{$D(e)$}
	arc[radius=0.4,start angle=0,end angle=180];
\draw[fill=white] (A)circle(2pt) (B)circle(2pt) (C)circle(2pt);
\end{tikzpicture}
\caption{Corner arc calculation}\label{fig-pos-gen}
\end{figure}
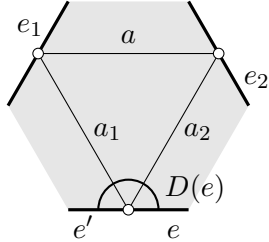

There are at least $(\hat{k}(e)\bmod{N})$ $-$ states on $e$. Since by definition, a leading term diagram is increasingly stated, the innermost $(\hat{k}(e)\bmod{N})$ copies of $D(e)$ must have $-$ states on $e$. On the other hand, the number of endpoints on $e'$ is $k(e)=k_\partial(e)\equiv0\pmod{N}$, and the number of $-$ states is $\hat{k}(e')\equiv-\hat{k}(e)\pmod{N}$. This means there are at least $(\hat{k}(e)\bmod{N})$ $+$ states on $e'$. Again by the increasing stated condition, the same copies of $D(e)$ must be assigned $+$ states on $e'$. Therefore, these copies of $D(e)$ are of the form $X_{\hat{e}}^{\hat{k}(e)\bmod{N}}$. Doing this for every boundary edge $e$, we see the diagram of $\lead(\alpha)$ contains the diagram of
\begin{equation}
\beta_1=\prod_{e\in\quasi_\partial}X_{\hat{e}}^{\hat{k}(e)\bmod{n}}.
\end{equation}
Now let $\alpha'$ be the diagram obtained by removing $\beta_1$ from the diagram of $\lead(\alpha)$, and let $\bar{k}'=\deg\alpha'$. The rest of the proof is essentially the same as the weaker version.

\textbf{Case 2} ($d=2$.) In this case $n=2N$ and $N$ is odd. Thus, we can consider $\bmod N$ and $\bmod2$ separately to obtain
\begin{equation}\label{eqn-center-2}
\begin{split}
k+P_+J\hat{k}&\equiv0\pmod{N},\\
P_+Jk_\partial&\equiv0\pmod{2},\\
k_\partial&\equiv0\pmod{N}.
\end{split}
\end{equation}
These relations have similar form as the previous case. Thus, the same argument produces the elements $\gamma'\in\skein_\epsilon(\surface)$ with degree $\bar{k}'/N$, $\beta$ a product of elements in $B_\zeta$ and $\gamma\in\skein_\zeta(\surface)$, $c\in\perialg$ with $c\Phi_\zeta(\gamma')=\gamma\beta$ and $\lead(\gamma)=\lead(\alpha)$.

By construction, $\bar{k}'$ satisfy the same relations \eqref{eqn-center-2}. 
If two boundary edges $b_1,b_2$ are connected by an edge $e\in\quasi$ as in Figure~\ref{fig-matching}, then by \eqref{eq-PJk},
\begin{equation}
k'(b_1)+k'(b_2)=(P_+Jk'_\partial)(e)\equiv0\pmod2.
\end{equation}
Then using connectedness, $k'(b_1)+k'(b_2)$ is even for any pair of boundary edges $b_1,b_2$. Since $N$ is odd, the same is true when divided by $N$. This means $\gamma'$ is matching, so $\gamma\in Z_\zeta$.

\textbf{Case 3} ($d=4$.) Since $n=4N$, the relations can be rewritten as
\begin{equation}\label{eqn-center-4}
\begin{split}
P_+Jk_\delta-k&\equiv0\pmod{2N},\\
k_\partial&\equiv0\pmod{2N},
\end{split}
\end{equation}
where $k_\delta=\frac{1}{2}k_\partial-\hat{k}$ is integral by the second relation. It is also half the boundary gradings. Following the same strategy, we multiply the first relation by $J^T$ to get
\begin{equation}
(J^T P_+J)k_\delta\equiv0\pmod{2N}.
\end{equation}
Since the modulus is even, there are more solutions than the previous cases.

Let $C_1,\dotsc,C_b$ be the boundary components of $\surface$. Suppose $C_j$ has $r_j$ ideal points. Order the edges of $C_j$ consecutively as $e_1,\dotsc,e_{r_j}$. As in the previous cases, $k_\delta(e_i)+k_\delta(e_{i+1})\equiv0\pmod{2N}$, and $k_\delta(e_i)\equiv(-1)^{r_j}k_\delta(e_i)\pmod{2N}$. If $r_j$ is odd, then $2k_\delta(e_i)\equiv0\pmod{2N}$. Thus, either $k_\delta(e_i)\equiv0\pmod{2N}$ for all $i$, or $k_\delta(e)\equiv N\pmod{2N}$ for all $i$. Let $\beta_j=1$ if the first condition is true, or let $\beta_j=X_{e_1}^N\dotsm X_{e_{r_j}}^N\in B_\zeta$ if the second is true. If $r_j$ is even, then the same construction in the previous cases with $\hat{k}$ replaced by $-k_\delta$ defines an element $\beta_j\in B_\zeta$. Finally, let $\beta_o$ be the product of all $\beta_j$ with $r_j$ odd, $\beta_e$ be the product of all $\beta_j$ with $r_j$ even, and $\beta=\beta_o\beta_e$.

Let $\bar{k}''=\deg(\alpha\beta_e)$, which is divisible by $N$ just like the previous cases. Let $\gamma''\in\skein_\epsilon(\surface)$ be the basis element without peripheral curves corresponding to the coloring $\bar{k}''/N$. Then $\lead(\alpha\beta_e)=\lead(c\Phi_\zeta(\gamma''))$ for some $c\in\perialg$. Again by construction, $\Phi_\zeta(\gamma'')$ contains a factor of $\beta_e$, so we can write $c\Phi_\zeta(\gamma'')=\gamma\beta_e$ for some $\gamma\in\skein_\zeta(\surface)$. It follows then $\lead(\gamma)=\lead(\alpha)$.

Note that $\beta_o$ has the form $\Phi_\zeta(\beta'_o)$ where $\beta'_o\in\skein_\epsilon(\surface)$ is the corresponding product of elements in $B_\epsilon$. Let $\gamma'=\gamma''\beta'_o$. Then
\begin{equation}
c\Phi_\zeta(\gamma')=c\Phi_\zeta(\gamma'')\Phi_\zeta(\beta'_o)
=\gamma\beta_e\beta_o=\gamma\beta.
\end{equation}
To show $\gamma\in Z_\zeta$, we just need to show that $\gamma'$ is even. Unlike the previous cases, $\gamma'$ is not necessary a single diagram because of the factor $\beta_o$. However, $\gamma'$ is still a product of basis elements, so it is homogeneous in both gradings. Therefore, it suffices to check that the leading term is even. By construction,
\begin{equation}
\bar{k}':=\deg(\gamma')=\frac{1}{N}(\deg\alpha+\deg\beta_o+\deg\beta_e)
\end{equation}
satisfies
\begin{equation}
k'_\delta\equiv0\pmod{2},\qquad\text{hence }
k'-P_+Jk'_\delta\equiv k'\equiv0\pmod{2}.
\end{equation}
The first implies that the boundary grading is divisible by 4. On the other hand, $k'\equiv0\pmod{2}$ translates to the even intersection condition. Thus, $\gamma'$ is even.
\end{proof}

\section{Dimension over the center}
\label{sec-dim}

\subsection{Statements of the results}

Assume $\surface$ is connected, triangulable, and have nonempty boundary. Define the following parameters.
\begin{enumerate}
\item $g$ is the genus of the surface.
\item $p=\abs{\peripheral}$ is the number of interior ideal points.
\item $v=\abs{\marked_\partial}$ is the number of boundary ideal points or boundary edges.
\item $b$ is the number of boundary components of $\surclose$.
\item $b_2$ is the number of boundary components with an even number of ideal points.
\end{enumerate}

\begin{lemma}
Let
\begin{equation}
r(\surface)=v-\chi(\surclose)=v+2g-2+b.
\end{equation}
Then
\begin{equation}
\abs{\bar{\quasi}}=3r(\surface)+2p,\qquad
\abs{\quasi}=\abs{\bar{\quasi}}-v.
\end{equation}
\end{lemma}

\begin{proof}
This is a standard Euler characteristic calculation. A triangulation is the 1-skeleton of a CW-complex structure of $\surclose$. It has $v$ vertices and $\abs{\quasi}$ edges. Let $f$ be the number of 2-cells, $p$ of which are monogons with the rest being triangles. Then
\begin{equation}
v-\abs{\quasi}+f=\chi(\surclose),\qquad
3(f-p)+p=2\abs{\quasi}-v.
\end{equation}
Solving the equations proves the lemma.
\end{proof}

By \cite[Proposition~4.4]{LeTriang}, $\skein_q(\surface)$ is a domain. Therefore, its center $Z$ has a field of fractions, denoted by $\tilde{Z}$. Then $\skein_q(\surface)\otimes_Z\tilde{Z}$ is a vector space over $\tilde{Z}$, whose dimension is called the \term{dimension} of $\skein_q(\surface)$ over its center and denoted $\dim_Z\skein_q(\surface)$.

\begin{lemma}
At a root of unity $q=\zeta$, $\skein_\zeta(\surface)$ is finitely generated as a module over its center $Z$. Consequently, $\dim_Z\skein_\zeta(\surface)$ is finite.
\end{lemma}

\begin{proof}
By \cite[Theorem~6.7]{LYSL2}, there exists one-component simple diagrams $\alpha_1,\dotsc,\alpha_s\in\skein_\zeta(\surface)$ such that elements of the form $\alpha_1^{i_1}\dotsm\alpha_s^{i_s}$ span $\skein_\zeta(\surface)$. Since $\Phi_\zeta(\alpha_j^d)$ is central and has degree $N\deg\alpha_j^d=\deg\alpha_j^n$, we can write
\begin{equation}
\alpha_j^n=c_j\Phi_\zeta(\alpha_j^d)+(\text{lower degree terms})
\end{equation}
where $c_j\in\perialg$, so first term is in $Z$. Therefore, an element $\alpha_1^{i_1}\dotsm\alpha_s^{i_s}$ with any exponent at least $n$ can be written as an element of $Z$ plus lower degree terms. Since the degrees are well-ordered, any element is a finite sum of elements of $Z$ and $\perialg$-multiples of $\alpha_1^{i_1}\dotsm\alpha_s^{i_s}$ with all exponents less than $n$. Therefore, as a $Z$-module, $\skein_\zeta(\surface)$ is generated by at most $n^s$ elements.
\end{proof}

\begin{theorem}\label{thm-dim-z}
At a root of unity $q=\zeta$, the dimension of $\skein_\zeta(\surface)$ over its center is
\begin{equation}
D_\zeta=\begin{cases}
N^{\abs{\bar{\quasi}}-b_2},&d=1,\\
2^{2\lfloor\frac{v-1}{2}\rfloor}N^{\abs{\bar{\quasi}}-b_2},&d=2,\\
2^{2g+2v-2}N^{\abs{\bar{\quasi}}-b_2},&d=4.
\end{cases}
\end{equation}
\end{theorem}

Although it is not immediately obvious, $\abs{\bar{\quasi}}-b_2$ is even. Thus, $D_\zeta$ is always a square.

\subsection{Proof of Theorem~\ref{thm-dim-z}}

Recall $\balLambda\subset\ints^{\bar\quasi}$ is the subgroup of balanced vectors. Let $\Lambda_\zeta\subset\balLambda$ be the subgroup of balanced solutions to \eqref{eqn-center}. Define the \term{residue group} at $q=\zeta$ to be
\begin{equation}
R_\zeta=\balLambda/\Lambda_\zeta.
\end{equation}

\begin{lemma}\label{lemma-res-size}
$\abs{R_\zeta}=D_\zeta$.
\end{lemma}

The proof of Lemma~\ref{lemma-res-size} is in the next section.

The degree can be considered as a map $\deg:\skein_\zeta(\surface)\setminus\{0\}\to\balLambda$. Clearly $\deg(Z)\subset \Lambda_\zeta$. By composing with the quotient map, we obtain $\deg_\zeta:\skein_\zeta(\surface)\setminus\{0\}\to R_\zeta$.

\begin{lemma}\label{lemma-deg-sur}
$\deg_\zeta$ is surjective.
\end{lemma}

\begin{proof}
The image of $\deg_\zeta$ is a submonoid of $R_\zeta$. Thus, it is also a subgroup since $R_\zeta$ is finite. By Lemma~\ref{lemma-bal-basis}, the image of $\deg$ generates $\balLambda$. Thus, the image of $\deg_\zeta$ generates $R_\zeta$, which implies $\deg_\zeta$ is surjective.
\end{proof}

By \cite[Corollary~5.2, Theorem~6.1]{FKLDimension}, Theorem~\ref{thm-dim-z} follows from Lemmas~\ref{lemma-res-size} and \ref{lemma-deg-sur}.

\subsection{Proof of Lemma~\ref{lemma-res-size}}

Let $\Lambda^\rmmatch\subset\balLambda$ be the subgroup of matching vectors, that is, balanced vectors $\bar{k}$ such that $k(e)\equiv k(e')\pmod{2}$ for any pair of boundary edges $e,e'$. 

Using the notations of Lemma~\ref{lemma-central-boundary}, for a boundary component $C$ with an even number $r$ of ideal points, let
\begin{equation}
\bar{k}_C=\sum_{i=1}^r(-1)^i\bar{k}_{\hat{e}_i},
\end{equation}
where $\bar{k}_{\hat{e}_i}\in\balLambda$ is the edge coloring of $X_{\hat{e}_i}$.
Up to $\bmod n$, this is the edge coloring of the first type of central elements in Lemma~\ref{lemma-central-boundary} with $k=n-1$. By construction, $\bar{k}$ is balanced and solves \eqref{eqn-center}. Let $\Lambda_B$ be the subgroup generated by all such vectors.

\begin{lemma}\label{lemma-balanced-index}
$\abs{\ints^{\bar{\quasi}}/\balLambda}=2^{\abs{\bar\quasi}-r(\surface)}$.
\end{lemma}

\begin{proof}
By Proposition~\ref{prob-bal}, $(2\ints)^{\bar\quasi}\subset\balLambda\subset\ints^\quasi\times(2\ints)^{\hat\quasi_\partial}$. Consider the quotient $\balLambda/(2\ints)^{\bar\quasi}$. By the second inclusion, each element of the quotient is uniquely dentermined by the $\bmod2$ reduction of the $\ints^\quasi$ components, which can be interpreted as a CW cochain $\quasi\to\ints/2$. Then the balanced condition translates exactly to the cocycle condition. Thus, we have an isomorphism
\begin{equation}
\balLambda/(2\ints)^{\bar\quasi}\cong Z^1(\surclose;\ints/2).
\end{equation}
Consider the coboundary map
\begin{equation}
\delta:C^0(\surclose;\ints/2)\to Z^1(\surclose;\ints/2),
\end{equation}
By definition,
\begin{equation}
\ker\delta=Z^0(\surclose;\ints/2)=H^0(\surclose;\ints/2),\quad \coker\delta=H^1(\surclose;\ints/2).
\end{equation}
Thus, over $\ints/2$,
\begin{equation}
\dim Z^1=\dim C^0-\dim H^0+\dim H^1=v-1+(2g+b-1)=r(\surface).
\end{equation}
Therefore,
\begin{equation}
\abs{\ints^{\bar\quasi}/\balLambda}
=\frac{\abs{\ints^{\bar\quasi}/(2\ints)^{\bar\quasi}}}{\abs{\balLambda/(2\ints)^{\bar\quasi}}}
=\frac{2^{\abs{\bar\quasi}}}{2^{r(\surface)}}.\qedhere
\end{equation}
\end{proof}

\begin{lemma}\label{lemma-bdry-summand}
$\Lambda_B$ is a direct summand of $\balLambda$.

If $d=1$, $\Lambda_\zeta=\Lambda_B+N\balLambda$. If $d=2$, $\Lambda_\zeta=\Lambda_B+N\Lambda^\rmmatch$.
\end{lemma}

\begin{proof}
By Lemma~\ref{lemma-bal-basis}, $\balLambda$ is free with basis $\bar{k}_a,a\in\bar\quasi$. For each boundary component $C$ with an even number of ideal points, we can replace one of the basis elements of $\balLambda$ with the corresponding generator of $\Lambda_B$. This shows that $\Lambda_B$ is a direct summand.

The last two statements follow from the argument in the proof of Theorem~\ref{thm-center-root-1}.
\end{proof}

\begin{lemma}
$\abs{\balLambda/\Lambda^\rmmatch}=2^{2\lfloor\frac{v-1}{2}\rfloor}$.
\end{lemma}

\begin{proof}
Consider the $\bmod{2}$ boundary grading map $\Delta_2:\balLambda\to(\ints/2)^{\quasi_\partial}$ given by
\begin{equation}
\Delta_2(\bar{k})(e)=\bar{k}(e)\bmod{2}.
\end{equation}
The image of $\Delta_2$ consists of vectors whose components sum to $0$, which is an index $2$ subgroup of $(\ints/2)^{\quasi_\partial}$. Thus, $\abs{\image\Delta_2}=2^{v-1}$.

Let $\mathbf{1}\subset(\ints/2)^{\quasi_\partial}$ be the subgroup generated by the all $1$ vector. Then $\Lambda^\rmmatch=\Delta_2^{-1}(\mathbf{1})$ by definition. Thus, the induced map
\begin{equation}
\bar{\Delta}_2:\balLambda/\Lambda^\rmmatch\to\image \Delta_2/(\image \Delta_2\cap\mathbf{1})
\end{equation}
is an isomorphism. The intersection $\image \Delta_2\cap\mathbf{1}$ is $\mathbf{1}$ when $v$ is even, and it is trivial when $v$ is odd. Thus, $\abs{\balLambda/\Lambda^\rmmatch}=2^{v-2}$ when $v$ is even, and $\abs{\balLambda/\Lambda^\rmmatch}=2^{v-1}$ when $v$ is odd.
\end{proof}

\begin{lemma}
When $d=4$, $\abs{\Lambda_\zeta/(n\ints)^{\bar{\quasi}}}=2^{\abs{\quasi}+b}N^{b_2}$.
\end{lemma}

\begin{proof}
It is clear that $(n\ints)^{\bar{\quasi}}\subset \Lambda_\zeta$, so the quotient makes sense. By Lemma~\ref{lemma-bal-basis}, a solution of the form $\bar{k}=(k,2\hat{k})$ to \eqref{eqn-center}, or equivalently \eqref{eqn-center-4}, is always balanced. Thus, we only need to count the $\bmod n$ solutions of the above form to \eqref{eqn-center-4}. The solutions can be generated in the following way.
\begin{enumerate}
\item For each boundary edge $e$, $k(e)=k_\partial(e)$ can be $0$ or $2N$.
\item For each boundary component, the boundary grading $2k_\delta=k_\partial+2\hat{k}$ is determined by a single edge, as argued in the proof of Theorem~\ref{thm-center-root-1}.
\begin{enumerate}
\item If there is an even number of ideal points, $2k_\delta$ can be any of the $2N$ even numbers for a fixed edge, and the rest are determined.
\item If there is an odd number of ideal points, $2k_\delta$ can be either $0$ or $2N$.
\end{enumerate}
\item For each boundary edge $e$, $2\hat{k}(e)$ is determined the chosen $k_\partial(e)$ and $2k_\delta(e)$.
\item For each internal edge $e$, $\bar{k}(e)\bmod{2N}$ is determined by the boundary gradings $k_\delta\bmod{2N}$. Thus, there are two choices $\bmod{n}$.
\end{enumerate}
Thus, the count is
\begin{equation}
\abs{\Lambda_\zeta/(n\ints)^{\bar{\quasi}}}
=2^v\cdot\left((2N)^{b_2}\cdot2^{b-b_2}\right)\cdot2^{\abs{\quasi}-v}
=2^{\abs{\quasi}+b}N^{b_2}.\qedhere
\end{equation}
\end{proof}

\begin{proof}[Proof of Lemma~\ref{lemma-res-size}]
When $d=1$,
\begin{equation}
R_\zeta\cong\frac{\balLambda/\Lambda_B}{\Lambda_\zeta/\Lambda_B}
=\frac{\balLambda/\Lambda_B}{N\balLambda/\Lambda_B}.
\end{equation}
By Lemma~\ref{lemma-bal-basis}, $\balLambda$ is a free abelian group of rank $\abs{\bar\quasi}$. Since $\Lambda_B$ is a direct summand of $\balLambda$ of rank $b_2$, $\balLambda/\Lambda_B$ is free with rank $\abs{\bar{\quasi}}-b_2$. Therefore, $\abs{R_\zeta}=N^{\abs{\bar{\quasi}}-b_2}=D_\zeta$.

When $d=2$, the same argument using the finite index subgroup $\Lambda^\rmmatch$ in place of $\balLambda$ shows $\abs{\Lambda^\rmmatch/\Lambda_\zeta}=N^{\abs{\bar{\quasi}}-b_2}$. Thus,
\begin{equation}
\abs{R_\zeta}=\abs{\balLambda/\Lambda^\rmmatch}\abs{\Lambda^\rmmatch/\Lambda_\zeta}
=2^{2\lfloor\frac{v-1}{2}\rfloor}N^{\abs{\bar{\quasi}}-b_2}=D_\zeta.
\end{equation}

When $d=4$, we have $(n\ints)^{\bar{\quasi}}\subset \Lambda_\zeta\subset\balLambda\subset\ints^{\bar{\quasi}}$. Then
\begin{align}
\abs{R_\zeta}
&=\frac{\abs{\ints^{\bar{\quasi}}/(n\ints)^{\bar{\quasi}}}}{\abs{\ints^{\bar{\quasi}}/\balLambda}\abs{\Lambda_\zeta/(n\ints)^{\bar{\quasi}}}}
=\frac{n^\abs{\bar{\quasi}}}{2^{\abs{\bar\quasi}-r(\surface)}\cdot2^{\abs{\quasi}+b}N^{b_2}}\\
&=2^{2g+2v-2}N^{\abs{\bar{\quasi}}-b_2}=D_\zeta.\qedhere
\end{align}
\end{proof}

\printbibliography

@article{FKLUnicity,
  title={Unicity for representations of the Kauffman bracket skein algebra},
  author={Frohman, Charles and Kania-Bartoszynska, Joanna and L{\^e}, Thang},
  journal={Inventiones mathematicae},
  volume={215},
  pages={609--650},
  year={2019},
  publisher={Springer}
}

@article{FKLDimension,
  title={Dimension and trace of the Kauffman bracket skein algebra},
  author={Frohman, Charles and Kania-Bartoszynska, Joanna and L{\^e}, Thang},
  journal={Transactions of the American Mathematical Society, Series B},
  volume={8},
  number={18},
  pages={510--547},
  year={2021}
}

@article{LYSL2,
  title={Quantum traces and embeddings of stated skein algebras into quantum tori},
  author={L{\^e}, Thang TQ and Yu, Tao},
  journal={Selecta Mathematica},
  volume={28},
  number={4},
  pages={66},
  year={2022},
  publisher={Springer}
}

@article{LeTriang,
  title={Triangular decomposition of skein algebras},
  author={L{\^e}, Thang TQ},
  journal={Quantum Topology},
  volume={9},
  number={3},
  pages={591--632},
  year={2018}
}

@article{BLFrob,
  title={The Chebyshev--Frobenius homomorphism for stated skein modules of 3-manifolds},
  author={Bloomquist, Wade and L{\^e}, Thang TQ},
  journal={Mathematische Zeitschrift},
  pages={1--43},
  year={2020},
  publisher={Springer}
}

@article{Pr,
  title={Fundamentals of Kauffman bracket skein modules},
  author={Przytycki, J{\'o}zef H},
  journal = {Kobe Journal of Mathematics},
  volume = {16},
  number = {1},
  pages = {45-66},
  year = {1999}
}

@inproceedings{Tu,
  title={Skein quantization of Poisson algebras of loops on surfaces},
  author={Turaev, Vladimir G},
  booktitle={Annales scientifiques de l'Ecole normale sup{\'e}rieure},
  volume={24},
  number={6},
  pages={635--704},
  year={1991}
}

@article{Bul,
  title={Rings of $SL_2(\mathbb{C})$-characters and the Kauffman bracket skein module},
  author={Bullock, Doug},
  journal={Commentarii Mathematici Helvetici},
  volume={72},
  number={4},
  pages={521--542},
  year={1997},
  publisher={Springer}
}

@article{BFK,
  title={Understanding the Kauffman bracket skein module},
  author={Bullock, Doug and Frohman, Charles and Kania-Bartoszy{\'n}ska, Joanna},
  journal={Journal of knot theory and its ramifications},
  volume={8},
  number={03},
  pages={265--277},
  year={1999},
  publisher={World Scientific}
}

@article{PS,
  title={On Skein Algebras And $SL_2(\mathbb{C})$-Character Varieties},
  author={Przytycki, J{\'o}zef H and Sikora, Adam S},
  journal = {Topology},
  volume = {39},
  number = {1},
  pages = {115-148},
  year = {2000}
}

@article{BWqtr,
  title={Quantum traces for representations of surface groups in $SL_2(\mathbb{C})$},
  author={Bonahon, Francis and Wong, Helen},
  journal={Geometry \& Topology},
  volume={15},
  number={3},
  pages={1569--1615},
  year={2011},
  publisher={Mathematical Sciences Publishers}
}

@article{Mu,
  title={Skein and cluster algebras of marked surfaces},
  author={Muller, Greg},
  journal={Quantum topology},
  volume={7},
  number={3},
  pages={435--503},
  year={2016}
}

@article{CL,
  title={Stated skein algebras of surfaces},
  author={Costantino, Francesco and L{\^e}, Thang TQ},
  journal={Journal of the European Mathematical Society},
  volume={24},
  number={12},
  pages={4063--4142},
  year={2022}
}

@article{Fa,
  title={Holonomy and (stated) skein algebras in combinatorial quantization},
  author={Faitg, Matthieu},
  journal={arXiv preprint arXiv:2003.08992},
  year={2020}
}

@article{Kor,
  title={Finite presentations for stated skein algebras and lattice gauge field theory},
  author={Korinman, Julien},
  journal={Algebraic \& Geometric Topology},
  volume={23},
  number={3},
  pages={1249--1302},
  year={2023},
  publisher={Mathematical Sciences Publishers}
}

@article{KQ,
  title={Classical shadows of stated skein representations at roots of unity},
  author={Korinman, Julien and Quesney, Alexandre},
  journal={arXiv preprint arXiv:1905.03441},
  year={2019}
}

@article{BWi,
  title={Representations of the Kauffman bracket skein algebra I: invariants and miraculous cancellations},
  author={Bonahon, Francis and Wong, Helen},
  journal={Inventiones mathematicae},
  volume={204},
  pages={195--243},
  year={2016},
  publisher={Springer}
}

@article{BWii,
  title={Representations of the Kauffman bracket skein algebra, II: Punctured surfaces},
  author={Bonahon, Francis and Wong, Helen},
  journal={Algebraic \& geometric topology},
  volume={17},
  number={6},
  pages={3399--3434},
  year={2017},
  publisher={Mathematical Sciences Publishers}
}

@article{BWiii,
  title={Representations of the Kauffman bracket skein algebra III: closed surfaces and naturality},
  author={Bonahon, Francis and Wong, Helen},
  journal={Quantum Topology},
  volume={10},
  number={2},
  pages={325--398},
  year={2019}
}

@article{KK,
  title={Classification of semi-weight representations of reduced stated skein algebras},
  author={Karuo, Hiroaki and Korinman, Julien},
  journal={arXiv preprint arXiv:2303.09433},
  year={2023}
}

@article{LYSurvey,
  title={Stated skein modules of marked 3-manifolds/surfaces, a survey},
  author={L{\^e}, Thang TQ and Yu, Tao},
  journal={Acta Mathematica Vietnamica},
  volume={46},
  number={2},
  pages={265--287},
  year={2021},
  publisher={Springer}
}

@article{LSStated,
  title={Stated SL (n)-skein modules and algebras},
  author={L{\^e}, Thang TQ and Sikora, Adam S},
  journal={arXiv preprint arXiv:2201.00045},
  year={2021}
}

@article{Hig,
  title={Triangular decomposition of $SL_3$ skein algebras},
  author={Higgins, Vijay},
  journal={Quantum Topology},
  volume={14},
  number={1},
  pages={1--63},
  year={2023}
}

@article{Wang,
  title={On stated $ SL (n) $-skein modules},
  author={Wang, Zhihao},
  journal={arXiv preprint arXiv:2307.10288},
  year={2023}
}

@book{BGQGroup,
  title={Lectures on algebraic quantum groups},
  author={Brown, Ken and Goodearl, Ken R},
  year={2012},
  publisher={Birkh{\"a}user}
}

@article{DL,
  title = {Quantum Function Algebra at Roots of 1},
  journal = {Advances in Mathematics},
  volume = {108},
  number = {2},
  pages = {205--262},
  year = {1994},
  author = {C. Deconcini and V. Lyubashenko}
}

@article{BY,
  title={Azumaya loci and discriminant ideals of PI algebras},
  author={Brown, Ken A and Yakimov, Milen T},
  journal={Advances in Mathematics},
  volume={340},
  pages={1219--1255},
  year={2018},
  publisher={Elsevier}
}

@inbook{DP,
  title={Quantum groups},
  booktitle={D-modules, representation theory, and quantum groups},
  author={De Concini, C and Procesi, C},
  pages={31--140},
  year={1992}
}

\end{document}